\documentclass{article}

\usepackage[utf8]{inputenc} % allow utf-8 input
\usepackage[T1]{fontenc}    % use 8-bit T1 fonts
\usepackage{lmodern}
\usepackage{booktabs}       % professional-quality tables
\usepackage{amsfonts}       % blackboard math symbols
\usepackage{nicefrac}       % compact symbols for 1/2, etc.
\usepackage{microtype}      % microtypography
\usepackage{graphicx}

\usepackage[round]{natbib}
\usepackage{caption}
\usepackage{libertine}
\usepackage{algorithm}
\usepackage{algorithmic}

\usepackage{array}
\usepackage{multirow}
\newcolumntype{M}[1]{>{\centering\arraybackslash}m{#1}}

\usepackage{tabularray}

\usepackage{makecell}
\usepackage{hhline}
\usepackage{array}
\usepackage{newfloat}
\usepackage{listings}
\DeclareCaptionStyle{ruled}{labelfont=normalfont,labelsep=colon,strut=off} % DO NOT CHANGE THIS
\floatstyle{ruled}
\newfloat{listing}{tb}{lst}{}
\floatname{listing}{Listing}
\usepackage{amsmath,amssymb,amsthm,mathtools}

\usepackage{xr-hyper}
\usepackage[capitalize,noabbrev]{cleveref}
\usepackage{enumitem}

\theoremstyle{plain}
\newtheorem{theorem}{Theorem}
\newtheorem{proposition}[theorem]{Proposition}
\newtheorem{lemma}[theorem]{Lemma}

\newtheorem{corollary}[theorem]{Corollary}

\theoremstyle{definition}
\newtheorem{definition}[theorem]{Definition}
\newtheorem{algorithm_def}[theorem]{Algorithm}

\theoremstyle{remark}
\newtheorem{remark}[theorem]{Remark}
\newtheorem{example}[theorem]{Example}

\newcommand{\R}{\mathbb{R}}
\newcommand{\N}{\mathbb{N}}
\newcommand{\ip}[2]{\langle #1,\, #2 \rangle}
\newcommand{\norm}[1]{\| #1 \|}
\newcommand{\innp}[2]{\langle #1, #2 \rangle}
\newcommand{\normW}[2]{\| #1 \|_{#2}}

\newcommand{\argmin}{\operatorname{arg\,min}}

\newcommand*{\cC}{\mathcal{C}}

\newcommand*{\RR}{\mathbb{R}}

\newcommand*{\NN}{\mathbb{N}}

\usepackage{booktabs}

\title{On the Iterate Convergence of AdaGrad for Generalized Smooth Convex Optimization}

\usepackage{authblk}

\author[1]{%
Mathieu Besançon\thanks{\texttt{mathieu.besancon@inria.fr}}%
}
\author[2]{%
Tung Quoc Le\thanks{\texttt{quoc-tung.le@univ-grenoble-alpes.fr}}%
}
\affil[2]{Université Grenoble Alpes, LJK, CNRS}
\affil[1]{Université Grenoble Alpes, Inria, LIG, CNRS}

\begin{document}

\maketitle

\begin{abstract}
We prove sequential convergence results for the AdaGrad algorithm family optimizing convex differentiable objectives. Specifically, we provide necessary and sufficient conditions for the convergence of iterates for the three main AdaGrad variants (AdaNorm, AdaDiag, AdaFull) when the objective is convex and locally Lipschitz-smooth, closing the question left open from the literature.
We harness this general result to study the three variants under the generalized $(L_0,L_1)$-smoothness condition and show sequential convergence for sufficiently small constant step size. Moreover, under the so-called $(L_0,L_1)$-polynomially modifiable smoothness assumption, which is a relaxation of the $(L_0,L_1)$ generalized smoothness property and is satisfied by many function classes such as $L$-smooth functions or univariate polynomials, sequential convergence for these AdaGrad variants is proved for arbitrary learning rates. This result provides conditions under which AdaGrad presents adaptivity, i.e., does not require tuning the parameters based on the instance.
Finally, we provide numerical illustrations of the behavior of AdaGrad on convex and nonconvex functions. In particular, we construct a counterexample empirically showing that smoothness alone is not sufficient for the sequential convergence of AdaGrad-type algorithms, and suggesting that additional geometric hypotheses (e.g., convexity as in this paper, or the Kurdyka-\L ojasiewicz inequality) are indispensable for sequential convergence results.
\end{abstract}

\section{Introduction}\label{sec:intro}
\paragraph{Setting} In this paper, we consider the unconstrained optimization problem:
\begin{equation*}
    \underset{x \in \RR^n}{\mathrm{Minimize}} \quad f(x),
\end{equation*}
where $f$ is convex and differentiable.
In contrast to a large body of existing literature, which assumes the global Lipschitz property of $\nabla f$, we focus on the case where $\nabla f$ is only locally Lipschitz (e.g., $(L_0,L_1)$-smoothness or variants) and the set of its minimizers $\argmin f \neq \emptyset$ (i.e., nonempty). We will investigate the \emph{sequential convergence} of the family of \emph{AdaGrad algorithms} \citep{duchi2011adaptive}, a staple of modern optimization for machine learning.

\paragraph{Why sequential convergence?} Theoretical guarantees for convex optimization algorithms take many different forms such as convergence in value, i.e., $f(x_k) \overset{k \to \infty}{\longrightarrow} \inf f$, convergence in gradient norm, i.e., $\|\nabla f(x_k)\| \overset{k \to \infty}{\longrightarrow} 0$ or sequential convergence (also known as iterate convergence), i.e., $\lim_{k \to \infty} x_k = x^\star \in \argmin f$ (assuming $\argmin f$ is  nonempty), as well as other quantitative counterparts. 

Mathematically, iterate convergence is the strongest notion since it implies convergence in value (if $f$ is continuous) and in gradient norm (if $f$ is $C^1$) while the converse implications do not necessarily hold.
Moreover, sequential convergence is an important measure for algorithmic stability and of practical interest in signal/image processing \citep{chaux2009nested} or statistics where we generally care more about the value of the minimizers than that of the training loss functions. 

Proving sequential convergence, however, is usually more involved. Indeed, even a natural question such as sequential convergence of gradient descent for convex smooth optimization remains open, despite its simplicity and popularity in many scientific domains.
Moreover, many other well-known algorithms, such as Gauss-Seidel (block-coordinate descent), gradient descent with line search, mirror descent \citep{bolte2020curiosity} and Frank-Wolfe \citep{bolte2024frankwolfe} do not converge in iterates in the smooth convex optimization setting.
Known results in this area usually combine geometric properties of the objective function (e.g., convexity, Kurdyka-\L ojasiewicz inequality, Lipschitz stratification) with the right assumptions on the algorithms (e.g., step sizes) to succeed \citep{attouch2005convergence,traore2021sequential,lai2026diametersubgradientsequencesominimal,chzhen2026convergenceratessubgradientdescent,absil2005convergence}. 

\paragraph{Contributions.}

\begin{table*}[htbp]
  \centering
  \begin{tblr}{
    hlines, vlines, rowsep = 4pt,
    colspec = {Q[c,m,wd=0.11\textwidth] Q[c,m,wd=0.22\textwidth]
               Q[c,m,wd=0.30\textwidth] Q[c,m,wd=0.26\textwidth]},
  }
    Algorithm / Setting & $L$-smooth (\Cref{def:L-smooth})
      & $(L_0,L_1)$ generalized smooth (\Cref{def:l0-l1-smooth})
      & $(L_0,L_1)$ polynomially modifiable smooth (\Cref{def:smoothness-polynomially-modifiable}) \\
    \ref{step:AdaNorm}
      & \SetCell[r=2]{c,m} \citet{traore2021sequential}: for any learning rate $\eta > 0$
      & \textbf{\Cref{thm:adanorm-main}}: for $0 < \eta \leq \frac{1}{2L_1}$
      & \SetCell[r=3]{c,m} \textbf{\Cref{theorem:adagrad-l0-l1-modifiable-convergence}}: for any $\eta > 0$ \\
    \ref{step:AdaDiag} &
      & \SetCell[r=2]{c,m} \textbf{\Cref{theorem:adagrad-l0-l1-convergence}}: for $0 < \eta \leq \eta^\star$ ($\eta^\star$ depends on $L_0,L_1$ and initialization)
      & \\
    \ref{step:AdaFull} & \textbf{\Cref{cor:adafull-L-smooth}}: for any $\eta > 0$ & & \\
  \end{tblr}
  \caption{Summary of sequential convergence results for the AdaGrad family. Our results are indicated with bold text.}
  \label{tab:4x3_table1}
\end{table*}

Our paper extends the understanding of sequential convergence of AdaGrad-type algorithms by proving the sequential convergence of three versions of AdaGrad algorithms, namely \ref{step:AdaNorm}, \ref{step:AdaDiag} and \ref{step:AdaFull} in a unified manner when minimizing Lipschitz-smooth or generalized smooth and convex functions.
In fact, it is known that \ref{step:AdaNorm} and \ref{step:AdaDiag}, \emph{with arbitrary step sizes}, can converge in iterates in smooth convex optimization. This shows the adaptivity of the method: one does not need to have access to the Lipschitz constant of the function gradient to achieve convergence properties. We take a step further, by proving the sequential convergence in the generalized smoothness setting. The new results are summarized in \cref{tab:4x3_table1}. Our contributions are given as follows:

\begin{enumerate}[leftmargin=*]
    \item We provide a necessary and sufficient condition for the sequential convergence of AdaGrad algorithms where $f$ is convex and $\nabla f$ is locally Lipschitz (cf. \Cref{theorem:unified-framework}).
    \item Harnessing \Cref{theorem:unified-framework}, we prove the sequential convergence of \ref{step:AdaNorm} under the assumption that $f$ is $(L_0, L_1)$ generalized smooth and convex and the learning rate $\eta \leq \frac{1}{2L_1}$ (cf. \Cref{thm:adanorm-main}). The same result is proved for \ref{step:AdaDiag} and \ref{step:AdaFull}, with $\eta \leq \eta^\star$, a threshold depending on $(L_0, L_1)$ and other parameters of the function (cf. \Cref{theorem:adagrad-l0-l1-convergence}).   
    \item We propose $(L_0,L_1)$ \emph{polynomially modifiable smoothness} (cf. \Cref{def:smoothness-polynomially-modifiable}), an intermediate property between $L$-smoothness and $(L_0,L_1)$ generalized smoothness which is satisfied by many $(L_0,L_1)$ smooth functions (e.g., univariate polynomial mapping, $\|\cdot\|_p^p$). Under this assumption, all three variants of AdaGrad converge in iterates, regardless of their step sizes (cf. \Cref{theorem:adagrad-l0-l1-modifiable-convergence}).
    \item We argue that smoothness may be insufficient to prove sequential convergence by providing an example of a bounded, $L$-smooth function whose AdaGrad iterates empirically fail to converge. Convexity, thus, is one of the sufficient conditions for sequential convergence.
\end{enumerate}

\paragraph{Related work.}
AdaGrad algorithms were extensively studied in the $L$-smooth case: \citet{rachel2019ada} showed that AdaGrad is not hyperparameter sensitive, and can match the performance of gradient descent without knowing the smoothness constant. If the function is also convex, \citet{traore2021sequential} further proved the sequential convergence. In the non-convex setting, subsequent works such as \citet{defossez2022a,zhou2024adagrad,kavis2022high,liu2023highprobability} focused on bounding the gradient norm in deterministic and stochastic settings. \citet{antonakopoulos2022adagrad} also proved that all three variants
of AdaGrad avoid saddle points, establishing preconditioner
stabilization and gradient summability under smoothness and
lower-boundedness (without convexity).

For $(L_0,L_1)$-smooth optimization,
\citet{ZhangKJ2020} showed that
gradient clipping accelerates training under such generalized
smoothness.
\citet{faw2023beyond}, \citet{bohan2023ada} and \citet{liu2025adagrad}
proved convergence \emph{rates}
$\min_{k\leq T}\norm{\nabla f(x_k)}^2 = \tilde{O}(1/T)$ (where $\tilde{O}(\cdot)$ hides logarithmic dependency) for \ref{step:AdaNorm} and \ref{step:AdaDiag}
in the stochastic nonconvex $(L_0,L_1)$-smooth setting and variants.
Our result establishes gradient \emph{summability} in the generalized smooth
and iterate \emph{convergence} in the convex setting, both of which are
qualitatively stronger.

\section{Problem formulations and mathematical tools}\label{sec:framework}

In this section, we introduce the AdaGrad family of optimization algorithms, the class of investigated objective functions and the mathematical tools that will be used to establish all the results.

\subsection{Variants of AdaGrad algorithms}
We first define the three AdaGrad variants through a common template in \cref{alg:adagradtemplate}.

\begin{algorithm_def}[AdaGrad template]\label{alg:adagradtemplate}
\begin{align}
& x_{k+1} = x_{k} - \eta\ G_k^{-\frac12} \nabla f(x_k) \tag{AdaGrad-Step}\label{step:AdaGrad-Step} \\
 & \text{where}  \nonumber \\
& G_k = \delta^2 I + \sum_{s=0}^k \norm{\nabla f(x_s)}^2  I \tag{AdaNorm} \label{step:AdaNorm} \\
& G_k = \delta^2 I + \sum_{s=0}^k \mathrm{diag}(\nabla f(x_s) \nabla f(x_s)^\top) \tag{AdaDiag} \label{step:AdaDiag} \\
& G_k =\delta^2 I + \sum_{s=0}^k \nabla f(x_s) \nabla f(x_s)^\top \tag{AdaFull} \label{step:AdaFull}
% & \Gamma_k = G_k^{-\frac12} \\
\end{align}
\end{algorithm_def}

\subsection{Assumptions on the objective functions}
In the following, we provide several definitions of smoothness, which will be used in this paper. We emphasize that no blanket assumptions will be made, and the assumptions for each theorem will be stated explicitly.
\begin{definition}[$L$-smooth]
\label{def:L-smooth}
$f$ is an $L$-smooth function if its gradient is $L$-Lipschitz, i.e.: $\|\nabla f(x) - \nabla f(y)\| \leq L \|x - y\|$.
\end{definition}
$L$-smoothness allows us to obtain the descent lemma:
\begin{align*}
  |f(y) - f(x) - \ip{\nabla f(x)}{y - x}|
  \leq \frac{L}{2}\norm{y - x}^2, \forall x, y \in \R^n.
\end{align*}
The above inequality shows that $f$ is upper-bounded by a quadratic function, which is not true for many practical cases.
\citet{ZhangKJ2020} proposes to study so-called $(L_0,L_1)$-generalized smooth functions, i.e.:
\begin{definition}[$(L_0,L_1)$ generalized smooth]
\label{def:l0-l1-smooth}
A function $f$ is $(L_0,L_1)$ (generalized) smooth if: 
$$\|\nabla^2 f(x)\| \leq L_0 + L_1 \|\nabla f(x)\|, \forall x \in \RR^n.$$
\end{definition}
In particular, an $L$-smooth function is an instance of $(L_0,L_1)$-smooth functions. We also have a generalized version of the descent lemma as well:
\begin{lemma}[Generalized descent lemma, (Lemma 2.5, \citet{vankov2025optimizing})]
\label{lem:gen-descent}
If $f$ is an $(L_0,L_1)$-smooth function with $L_1 > 0$, for all $x, y \in \R^n$,
\begin{align}
  & \bigl|f(y) - f(x) - \ip{\nabla f(x)}{y-x}\bigr|
  \leq \\
  & \;\;\;\;\; \frac{L_0 + L_1\norm{\nabla f(x)}}{L_1^2}\,
  \varphi\bigl(L_1\norm{y-x}\bigr),\nonumber
\end{align}
where $\varphi(t) = e^t - t - 1$ satisfies
$\varphi(t) \leq t^2 e^t/2$\, for all $t \geq 0$.
\end{lemma}

To ease the presentation, we will use the notations
\begin{align*}
& g_k = \nabla f(x_k) \quad, \quad S_T = \sum_{s=0}^T \norm{g_s}^2,
\end{align*}
in the rest of the paper.

\begin{theorem}[Necessary and sufficient conditions for sequential convergence of AdaGrad algorithms in the convex locally smooth setting]
\label{theorem:unified-framework}
Let $f$ be a convex function with locally Lipschitz gradient and nonempty $\argmin f$.
Consider the three AdaGrad variants of \cref{alg:adagradtemplate}, the following three statements are equivalent:
\begin{enumerate}[leftmargin=*]
    \item The sequence $\{x_k\}_{k \in \NN}$ converges to an element of $\argmin f$.
    \item The sequence $\{x_k\}_{k \in \NN}$ is bounded.
    \item The sequence $\{\|\nabla f(x_k)\|^2\}_{k \in \NN}$ is summable.
\end{enumerate}
\end{theorem}

The proof is available in the supplementary materials and shows that the statements hold
through a cycle of implications:
\begin{align*}
1) \implies 2) \implies 3) \implies 1).   
\end{align*}

We emphasize that this framework covers functions whose derivatives are simply locally Lipschitz, and no global smoothness in the sense of \Cref{def:L-smooth} is needed. It allows us to cover a broader class of functions such as $(L_0,L_1)$-smooth functions (cf. \Cref{def:l0-l1-smooth}). 

In addition, \Cref{theorem:unified-framework} establishes the equivalence between sequential convergence and the summability of $\{ \norm{\nabla f(x_k)}^2 \}_{k \in \NN}$. Therefore, in the remainder of the paper, all of our proofs for sequential convergence consist in proving the finiteness of the sum $\sum_{k \in \NN} \|\nabla f(x_k)\|^2$. 

\section{Convergence of AdaGrad under $(L_0,L_1)$-smoothness}
In this section, we discuss the sequential convergence of the AdaGrad algorithms i.e., \cref{alg:adagradtemplate}. Our result is given as follows:

\begin{theorem}[Sequential convergence of \ref{step:AdaNorm} under $(L_0,L_1)$-smoothness]
\label{thm:adanorm-main}
Let $f$ be a convex, $(L_0,L_1)$-smooth function with $L_1 > 0$ and $\argmin f \neq \emptyset$. Then, for any $\delta > 0$ and
$\eta \leq 1/(2L_1)$, the sequence $\{x_k\}_{k\in\N}$ produced by
\ref{step:AdaNorm} converges to a point in $\argmin f$.
\end{theorem}
\begin{proof}
For $\eta \leq \frac{1}{2L_1}$, we have:
\begin{equation*}
    \|x_{k + 1} - x_k\| \leq \eta\left\|\frac{g_k}{\sqrt{\delta^2 + S_k}}\right\| \leq \eta \leq \frac{1}{2L_1}.
\end{equation*}
By the generalized descent lemma (cf. \Cref{lem:gen-descent}):
\begin{equation*}
    \begin{aligned}
        f(x_{k+1}) - f(x_k) \leq &- \eta \frac{\|g_k\|^2}{\sqrt{S_k + \delta^2}}\\ + &\sqrt{e}\eta^2\frac{L_0 + L_1\|g_k\|}{2}\frac{\|g_k\|^2}{\delta^2 + S_k}.
    \end{aligned}
\end{equation*}
Denote $C_1 = f(x_0) - \inf f$, telescoping gives us:
\begin{equation}
    \label{eq:telescope-adanorm-l0-l1}
    \eta \sum_{k = 0}^T \frac{\|g_k\|^2}{\sqrt{\delta^2 + S_k}} \leq C_1 + \sqrt{e}\eta^2 \sum_{k = 0}^T \frac{L_0 + L_1 \|g_k\|}{2} \frac{\|g_k\|^2}{S_k + \delta^2}.
\end{equation}
Since $\|g_k\| < \sqrt{S_k + \delta^2}$, we obtain:
\begin{equation*}
    \frac{\|g_k\|^3}{\delta^2 + S_k} \leq \frac{\|g_k\|^2}{\sqrt{S_k + \delta^2}}, \forall k \in \NN.
\end{equation*}
Applying this inequality, one obtains:
\begin{equation}
    \label{eq:telescope-L-smooth-alike-1}
    \underbrace{\eta\left(1 - \frac{\sqrt{e}\eta L_1}{2}\right)}_{>0} \sum_{k = 0}^T \frac{\|g_k\|^2}{\sqrt{\delta^2 + S_k}} \leq C_1 + \frac{L_0\sqrt{e}\eta^2}{2} \sum_{k = 0}^T \frac{\|g_k\|^2}{S_k + \delta^2}.
\end{equation}
We estimate the summations in both sides of \eqref{eq:telescope-L-smooth-alike-1}:
\begin{equation*}
    \sum_{k = 0}^T \frac{\|g_k\|^2}{\sqrt{\delta^2 + S_k}} \geq \frac{1}{\sqrt{S_T + \delta^2}}\sum_{k = 0}^T \|g_k\|^2 = \frac{S_T}{\sqrt{S_T + \delta^2}}.
\end{equation*}
\begin{equation*}
    \begin{aligned}
        \sum_{k = 0}^T \frac{\|g_k\|^2}{S_k + \delta^2} \leq \sum_{k = 0}^T \int_{S_{k-1}}^{S_k} \frac{dt}{t + \delta^2} = \int_{0}^{S_T} \frac{dt}{t + \delta^2} = \ln \frac{\delta^2 + S_T}{\delta^2}.
    \end{aligned}
\end{equation*}
Combining the last three inequalities, we obtain that $S_T$ is bounded, since the L.H.S of \eqref{eq:telescope-L-smooth-alike-1} having $O(\sqrt{S_T})$ growth while the R.H.S. of \eqref{eq:telescope-L-smooth-alike-1} grows like $O(\ln S_T)$. \Cref{theorem:unified-framework} allows us to conclude the proof.
\end{proof}

We briefly highlight links of the proof technique to the literature.
Estimating $\sum_{k\in\NN}\|g_k\|^2$ using different growth in the L.H.S and R.H.S of \eqref{eq:telescope-L-smooth-alike-1} is known (e.g., \citet{bohan2023ada}).
We simplify and adapt this technique in the deterministic setting. Comparing to other (deterministic) approaches, e.g., \citet{rachel2019ada,traore2021sequential} where one handles the first index $k_0$ where $S_{k_0}$ is bigger than the smoothness constant $L$, our manipulation is simpler, can be extended to other variants, and is suitable for pedagogical and presentation purposes.

In comparison to $L$-smooth cases, we have an extra condition on the learning rate $\eta \leq \frac{1}{2L_1}$, which is standard when dealing with $(L_0,L_1)$-smoothness. 
We do not know whether the condition $\eta \leq \frac{1}{2L_1}$ can be relaxed \footnote{In fact, \citet[Theorem 9]{bohan2023ada} states that if $\eta \geq \frac{9\sqrt{5}}{2L_1}$, the sequence of iterates might diverge for an adversarial $(L_0,L_1)$-smooth function. However, since their adversarial function is nonconvex,
it would not directly apply to our setting. Moreover, they also used an extension theorem in their construction, which is not known to preserve the $(L_0,L_1)$-generalized smoothness.}. 

We present in the following our results for \ref{step:AdaDiag} and \ref{step:AdaFull}.
\begin{theorem}[Sequential convergence of \ref{step:AdaDiag} and \ref{step:AdaFull} under $(L_0,L_1)$ smoothness]\label{theorem:adagrad-l0-l1-convergence}
Let $f:\RR^n \to \RR$ be a convex, $(L_0,L_1)$-smooth function with nonempty $\argmin f$.
For any $\delta > 0$ and $\mathcal{C} \subseteq \RR^n$ a compact subset of $\RR^n$, there exists a step-size threshold $\eta^\star := \eta^\star(\delta, \mathcal{C},L_0,L_1) > 0$ such that if $\eta < \eta^\star$ and $x_0 \in \mathcal{C}$, the iterates $\{x_k\}_{k \in \NN}$ generated by 
\ref{step:AdaDiag} (resp. \ref{step:AdaFull}) converge to a minimizer $x^\star \in \argmin f \neq \emptyset$. 
\end{theorem}
The proof of \Cref{theorem:adagrad-l0-l1-convergence} is significantly different from \Cref{thm:adanorm-main}, and is given in the supplementary materials. We provide a sketch of its proof in the following.
\begin{proof}[Sketch proof of \cref{theorem:adagrad-l0-l1-convergence}]
    We will deal with \ref{step:AdaDiag} in this proof. \ref{step:AdaFull} can be treated similarly and its proof is provided in the supplementary materials. 
    
    If we constrain $\eta^\star \leq \frac{\ln 2}{L_1\sqrt{n}}$, we get a similar descent lemma whose telescope resembles \Cref{eq:telescope-adanorm-l0-l1} as follows: 
    \begin{equation*}
        \eta \sum_{k = 0}^T  \underbrace{g_k^\top G_k^{-1/2} g_k}_{\Psi_k} \leq C_1 + \eta^2 \sum_{k = 0}^T (L_0 + L_1\|g_k\|) \underbrace{g_k^\top G_k^{-1} g_k}_{\Phi_k},
    \end{equation*}
    where $C_1:= {f(x_0) - f^\star}$. Thus, defining $\Delta(\mathcal{C}):= \max_{x_0 \in \mathcal{C}} f(x_0) - f^\star$, we deduce that with an initialization $x_0 \in \mathcal{C}$, one has:
    \begin{equation}
        \label{eq:telescope-abstract}
        \eta \sum_{k = 0}^T \Psi_k \leq \Delta(\mathcal{C}) + \eta^2 \sum_{k = 0}^T (L_0 + L_1\|g_k\|) \Phi_k.
    \end{equation}
    
    However, the rest of the proof cannot use the same approach as the proof of \Cref{thm:adanorm-main} due to the fact that:
    \begin{equation*}
        \|g_k\|g_k^\top G_k^{-1}g_k \nleq g_k^\top G_k^{-1/2}g_k.
    \end{equation*}
    The reason is that the elements of the diagonal matrix $G_k$ of \ref{step:AdaDiag} are the sums of the squares of the coordinates of $g_k$, and thus, the bound is false in general for \ref{step:AdaDiag} (and also \ref{step:AdaFull}).
        
    Instead, we will use a different strategy to prove the finiteness of $S_T$. We argue that if $\eta$ is chosen small enough, there exists an interval $[s_1, s_2]$ such that:
    \begin{enumerate}[leftmargin=*]
        \item $s_1 > S_0 = \|\nabla f(x_0)\|^2$.
        \item If $S_T \in [s_1, s_2]$, then \eqref{eq:telescope-abstract} cannot hold. Therefore, $S_T \notin [s_1, s_2], \forall T \in \NN$.
        \item The right-end point is much bigger than the left-end one in the sense that: $s_2 > 10s_1 + 2\frac{L_0^2}{L_1^2}$. This ensures that if $k \geq 1$ is the first index where $S_{k} \geq s_1$, then $S_k \in [s_1, s_2]$, which causes a contradiction.
    \end{enumerate}
    Thus, one can conclude that $S_T \leq s_1, \forall T \in \NN$. Our construction of $s_1$ and $s_2$ is given by:
    \begin{equation}
        \label{eq:s1-s2-construction}
        \begin{aligned}
            s_1 &= \left(\frac{8 \Delta(\mathcal{C})}{\eta}\right)^2 + \left(\frac{(32n\eta L_0)^4}{\delta^2} + \delta^2\right) + s_{\min},\\
            s_2 &= \delta^2 \left(\exp\left(\frac{1}{(4\eta L_1 n)^2}\right) - 1\right).\\
        \end{aligned}
    \end{equation}
    where $s_{\min}:= 10(B(\cC)+ \delta^2 / 3) + 2\frac{L_0^2}{L_1^2}$
    with the constant $B(\cC):= \max_{x_0 \in \mathcal{C}} \|\nabla f(x_0)\|^2$. 
    
    Since $s_1$ is polynomial w.r.t. $1/\eta$ and $s_2$ is exponential w.r.t. $1/\eta^2$, one has $s_2 \gg s_1$ when $\eta$ is small enough, concluding the proof.
    Detailed computations are available in the supplementary materials.
\end{proof} 

\begin{remark}
    In contrast to \Cref{thm:adanorm-main}, \Cref{theorem:adagrad-l0-l1-convergence} requires the step size threshold $\eta^\star$ to be dependent on $L_0, L_1, \delta$ and the set of initialization $\cC$, which is more restrictive. Nevertheless, to the best of our knowledge, our result is the first one to explicitly establish the finiteness of $\sum_k \|g_k\|^2$ under $(L_0,L_1)$-smoothness. Note that this result also implies that $\min_{0 \leq k \leq T} \|g_k\|^2 = O(1/T)$ for sufficiently small (but constant) learning rate $\eta$. 

    Moreover, one can explicitly estimate the value of the learning rate threshold $\eta^\star$ as follows.
    Define
    $\Delta(\cC) = \max_{x \in \cC} f(x) - f^\star$, $B(\cC) = \max_{x \in \cC} \|\nabla f(x)\|^2$, and $b=640\Delta(\cC)^2 / \delta^2$.
     Consider two algorithms:
    \begin{enumerate}[leftmargin = *]
        \item For \ref{step:AdaDiag}:
        \begin{align*}
            c &= \frac{10(32\sqrt{n} L_0)^4}{(L_1\delta)^4} + \frac{100 B(\cC)}{\delta^2} + \frac{22L_0^2}{\delta^2L_1^2} + 45, \\
            \eta^\star &= \frac{1}{4nL_1\sqrt{\ln(1 + 2c + 4096(b+1)^2n^4L_1^4)}}.
        \end{align*}
        \item For \ref{step:AdaFull}, we have:
        \begin{align*}
            c &= \frac{10(32{n} L_0)^4}{(L_1\delta)^4} + \frac{100 B(\cC)}{\delta^2} + \frac{22L_0^2}{\delta^2L_1^2} + 45, \\ 
            \eta^\star &= \frac{1}{4\sqrt{n}L_1\sqrt{\ln(1 + 2c + 4096(b+1)^2n^2L_1^4)}}.
        \end{align*}
    \end{enumerate}
    Thus, the step size threshold of \ref{step:AdaFull} has a dependency $1/\sqrt{n}$, while that of \ref{step:AdaDiag} is $1/n$ (where we recall that $n$ is the dimension of the variable $x$). The details of our derivations are given in the supplementary materials. 
\end{remark}

Lastly, we emphasize that the sufficiently small step size in our proof technique is different from those used in the literature, e.g., \citet{bolte2025inexact,josz2024global}:
\begin{enumerate}[leftmargin=*]
    \item Small step sizes allow one to obtain a (quasi)-descent phenomenon \citep{bolte2025inexact,josz2024global}. In our case, AdaGrad algorithms do not necessarily decrease the values of $f$.
    \item Coupling this phenomenon with the coercivity of $f$ allows one to prove that $x_k$ is bounded, which is equivalent to sequential convergence in our setting (cf. \Cref{theorem:unified-framework}). However, we do not assume that function $f$ to be coercive. 
\end{enumerate}
\section{Convergence of AdaGrad under $(L_0,L_1)$ polynomially modifiable smoothness}
It is shown in \Cref{thm:adanorm-main} and \Cref{theorem:adagrad-l0-l1-convergence} that under $(L_0, L_1)$ generalized smoothness, sequential convergence does hold when the learning rate is smaller than a certain threshold. These results are less satisfying than those under $L$-smoothness assumption, i.e., the AdaGrad family converges under any learning rate (\cref{theorem:unified-framework} and \cref{cor:adafull-L-smooth}).

In this section, we introduce a novel class of smooth functions --- $(L_0,L_1)$-polynomially modifiable smooth --- under which we can recover the sequential convergence for AdaGrad algorithms with arbitrary $\eta > 0$.

Intuitively, this class is a subset of $(L_0,L_1)$-smooth functions whose smooth constants $(L_0,L_1)$ can be varied such that $L_1$ can be chosen arbitrarily close to zero and $L_0$ has to be modified w.r.t. $L_1$ in a way that it is always bounded by $O(\mathtt{poly}(1/L_1))$. Its formal definition is given in \Cref{def:smoothness-polynomially-modifiable}. 
As we will explain, most of the known examples of $(L_0,L_1)$-smooth functions (except the exponential function) listed in \cite{ZhangKJ2020} are instances of this class.

\begin{definition}[$(L_0,L_1)$-polynomially modifiable smoothness]
    \label{def:smoothness-polynomially-modifiable}
    A $C^2$ function $f: \RR^n \to \RR$ is called $(L_0,L_1)$ polynomially modifiable smooth if there exists a sequence of positive scalars $\{\ell_k\}_{k \in \NN}$ and a polynomial $P:\RR_+ \to \RR_+$ such that:
    \begin{enumerate}
        \item $\lim_{k \to \infty} \ell_k = 0$.
        \item The function $f$ is $(P(1 / \ell_k), \ell_k)$-smooth, i.e.,
        \begin{equation*}
            \|\nabla^2 f(x)\| \leq P(1 / \ell_k) + \ell_k\|\nabla f(x)\|, \forall x \in \RR^n.
        \end{equation*}
    \end{enumerate}
\end{definition}

\Cref{def:smoothness-polynomially-modifiable} is motivated from the fact that many $(L_0,L_1)$ smooth functions are instances of \Cref{def:smoothness-polynomially-modifiable}.

\begin{example}
    \label{ex:L-smooth-is-good}
    An $L$-smooth function is also $(L_0,L_1)$ polynomially modifiable since one can take $P(1/\ell_k) = L$, the constant function.  
\end{example}

\begin{example}
    \label{ex:univariate-poly-is-good}
    A typical example of an $(L_0,L_1)$-smooth function is a univariate polynomial $R: \RR \to \RR$ with degree $\deg(R) \geq 2$. We show that $R$ is also an instance of \Cref{def:smoothness-polynomially-modifiable}. Indeed, since $Q:= \nabla R$ and $Q':= \nabla^2 R$ are also polynomial, by exploiting the factorization of $Q = a\prod_{i = 1}^{\deg(Q)} (x - z_i)$ (with potentially complex roots), we have:
    \begin{equation*}
        \begin{aligned}
            \left|\frac{\nabla^2 R(x)}{\nabla R(x)}\right| &= \left|\frac{Q'(x)}{Q(x)}\right| = \left|\sum_{i = 1}^{\deg(Q)} \frac{1}{x - z_i}\right|\\
            &\leq \sum_{i = 1}^{\deg(Q)} \frac{1}{|x - z_i|} \leq \sum_{i = 1}^{\deg(Q)} \frac{1}{|x| - |z_i|},
        \end{aligned}
    \end{equation*}
    for $|x| > z:= \max_{i = 1, \ldots, \deg(Q)} |z_i| \geq 0$. Thus, if $|x| \geq z + \frac{\deg(Q)}{\ell_k}$, we conclude that:
    \begin{equation*}
        \left|\frac{\nabla^2 R(x)}{\nabla R(x)}\right| \leq \deg(Q) \frac{\ell_k}{\deg(Q)} = \ell_k.
    \end{equation*}
    By choosing $P(1 / \ell_k) = \max_{|x| \leq z + \deg(Q)/\ell_k} |\nabla^2 R(x)|$, we ensure that:
    \begin{equation*}
        |\nabla^2 R(x)| \leq \ell_k|\nabla R(x)| + P(1/\ell_k),
    \end{equation*}
    for any $\ell_k > 0$. Note that $P(\cdot)$ is non-decreasing on $\RR_+$ and is a semi-algebraic function (since it is expressed as a maximum of a semi-algebraic function $|\nabla^2 R(x)|$). Therefore, it is polynomially bounded (i.e., there exists an integer $N > 0$ such that for sufficiently large $x > 0$, $P(x) < x^N$). To have a uniform bound on $\RR_+$, it is sufficient to choose a large enough constant $C$ such that $P(x) \leq x^N + C, \forall x \in \RR_+$.
\end{example}

\begin{example}
    \label{ex:Lp-norm-is-good}
    Another typical example of $L_0,L_1$-smooth function is $f(x) = \frac{1}{p} L_p^p(x) := \frac{1}{p}\sum_{i = 1}^n |x_i|^p, p \geq 2$.
    The gradient and Hessian of $f$ are given by:
    \begin{align*}
        & \nabla f(x) = \begin{pmatrix}
            \mathrm{sign}(x_1) |x_1|^{p - 1} \\ \ldots \\ \mathrm{sign}(x_n) |x_n|^{p - 1}
        \end{pmatrix} \\
        & \nabla^2 f(x) = (p-1) \begin{pmatrix}
            x_1^{p - 2} & \ldots & 0 \\
            \vdots & \ddots & \vdots \\
            0 & \ldots & x_n^{p - 2}
        \end{pmatrix}
    \end{align*}
    In each coordinate, we have:
    \begin{equation*}
        (p - 1)\left|x_i^{p - 2}\right| \leq \ell_k|x_i^{p - 1}| + \underbrace{\left(\frac{p - 1}{\ell_k}\right)^{p - 2}}_{P(1/\ell_k)}, \forall \ell_k > 0, x_i \in \RR.
    \end{equation*}
    As a consequence, we have:
    \begin{equation*}
        \begin{aligned}
            \|\nabla^2 f(x)\| &= \max_{i = 1, \ldots, n} (p-1)|x_i^{p-2}|\\
            &\leq \max_{i = 1, \ldots,n} \ell_k |x_i^{p - 1}| + P(1/\ell_k) \\
            &\leq \ell_k \|\nabla f(x)\| + P(1/\ell_k).
        \end{aligned}
    \end{equation*}
\end{example}

Thus, one can see that the class of $(L_0,L_1)$-polynomially modifiable smooth is rich and contains many typical examples of $(L_0,L_1)$-smooth functions (with some exceptions, such as the exponential function). In the following, we show that AdaGrad algorithms can converge in iterates when minimizing this class of functions with \emph{arbitrary step size}, just like the $L$-smooth setting in \cite{traore2021sequential}.

\begin{theorem}[AdaGrad sequential convergence under \Cref{def:smoothness-polynomially-modifiable}]
    \label{theorem:adagrad-l0-l1-modifiable-convergence}
    Let $f: \RR^n \to \RR$ be a convex, $(L_0,L_1)$ polynomially modifiable smooth function with nonempty $\argmin f$. Any sequence of iterates generated by \ref{step:AdaNorm}, \ref{step:AdaDiag}, or \ref{step:AdaFull} with $\delta >0$ and $\eta > 0$ converges to $x^\star \in \argmin f$.
\end{theorem}
\begin{proof}
    For \ref{step:AdaNorm}, the results are immediate since one can take $L_1$ arbitrarily small to satisfy $\eta \leq \frac{1}{2L_1}$.
    
    It remains to prove the result for \ref{step:AdaDiag} (and \ref{step:AdaFull}). The proof is identical to that of \Cref{theorem:adagrad-l0-l1-convergence}. The only difference is that one can no longer choose the learning rate $\eta$ sufficiently small to ensure that $\|\nabla f(x_k)\|^2$ is summable. 

    We can, however, modify the constants $(L_0, L_1)$ so that the same principle works. In particular, it is sufficient to show that we can choose $L_0, L_1$ such that: 
    \begin{enumerate}[leftmargin=*]
        \item $\eta \leq \frac{\ln 2}{L_1\sqrt{n}}$: This will be always possible, since $L_1$ can be taken arbitrarily close to zero due to \Cref{def:smoothness-polynomially-modifiable}.
        \item $s_2 \geq 10 s_1 + 2\frac{L_0^2}{L_1^2}$ (where $s_1, s_2$ are defined by \Cref{eq:s1-s2-construction}). The other conditions (i.e., $s_1 \geq s_{\min}$ and $h(x) > 0, \forall x \in [s_1,s_2]$) are satisfied automatically by construction or by the previous reasoning. 
        
        This is where one fully harnesses \Cref{def:smoothness-polynomially-modifiable}: we observe that $s_2$ is a exponential w.r.t. to $1 / L_1$, where $10s_1 + 2\frac{L_0^2}{L_1^2}$ is at most polynomial w.r.t. $1/L_1$ (because $L_0$ can be taken as $P(1/L_1)$ by \Cref{def:smoothness-polynomially-modifiable}). Therefore, with small enough $L_1$ (which is possible since $\lim_k \ell_k = 0$ as in \Cref{def:smoothness-polynomially-modifiable}), the condition $s_2 \geq 10 s_1 + 2\frac{L_0^2}{L_1^2}$ is satisfied and we conclude the proof.  
    \end{enumerate}
\end{proof}
Since an $L$-smooth function is also $(L_0,L_1)$ polynomially modifiable smooth, we have the following immediate result:
\begin{corollary}[Sequential convergence of \ref{step:AdaFull}]
    \label{cor:adafull-L-smooth}
    Let $f$ be a convex, $L$-smooth function with nonempty $\argmin f$. Any sequence of iterates generated by \ref{step:AdaFull} with $\delta > 0$ and $\eta > 0$ converges to a minimizer of $f$. 
\end{corollary}
In particular, \Cref{cor:adafull-L-smooth} completes the sequential convergence result for the AdaGrad family on $L$-smooth functions, left open since \citet{traore2021sequential}.

\section{Numerical illustrations}

We now provide numerical illustrations of AdaGrad algorithms.
The source code is available in the supplement.
We first start with the function:
\begin{align*}
f^{\mathrm{seg}}(x, y) = \left( \max\left\{ |x| -\frac12, 0\right\} \right)^4 + 2 y^4
\end{align*}
which is convex, $C^3$, and has $[-\frac12, \frac12] \times \{0\}$ as its set of minimizers.
It is also $(L_0,L_1)$ polynomially modifiable but not $L$-smooth, since its growth cannot be bounded by a quadratic function, thus illustrating one of the settings investigated in this work.
In \cref{fig:contoursegadagrad}, we represent the trajectories of the three AdaGrad variants on $f^{\mathrm{seg}}$ with step size $\eta=10$.
All algorithms converge to one of the two end points of the optimal segment.
A key observation is the two-phase behavior of the algorithm: the first two iterations overshoot the ideal step size (i.e., line search),
after which the conditioning matrix has collected enough information to stabilize the direction and magnitude of the updates.
\begin{figure}[ht]
    \centering
    \includegraphics[width=1.1\linewidth]{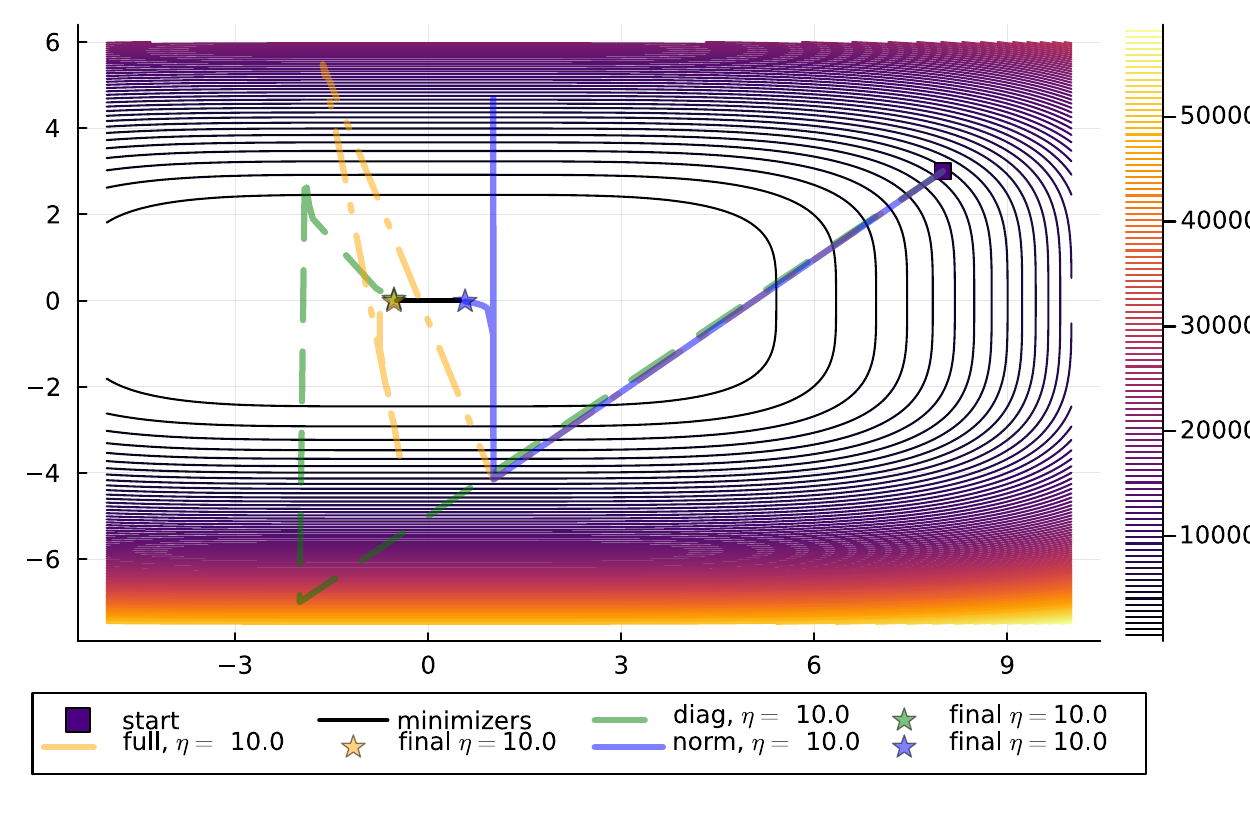}
    \caption{Trajectory of the three AdaGrad algorithms on the function $f^{\mathrm{seg}}$.}
    \label{fig:contoursegadagrad}
\end{figure}
\cref{fig:contoursegadadiag} provides a complementary view with the trajectory of AdaDiag, the main AdaGrad variant, with different step sizes.
At $\eta=1$, the algorithm (slowly) converges to the left-most point of the optimal segment while more aggressive step sizes $10$ and $20$ overshoot the
optimal segment in the first iteration before correcting and converging towards a single point.
\begin{figure}[ht]
    \centering
    \includegraphics[width=1.0\linewidth]{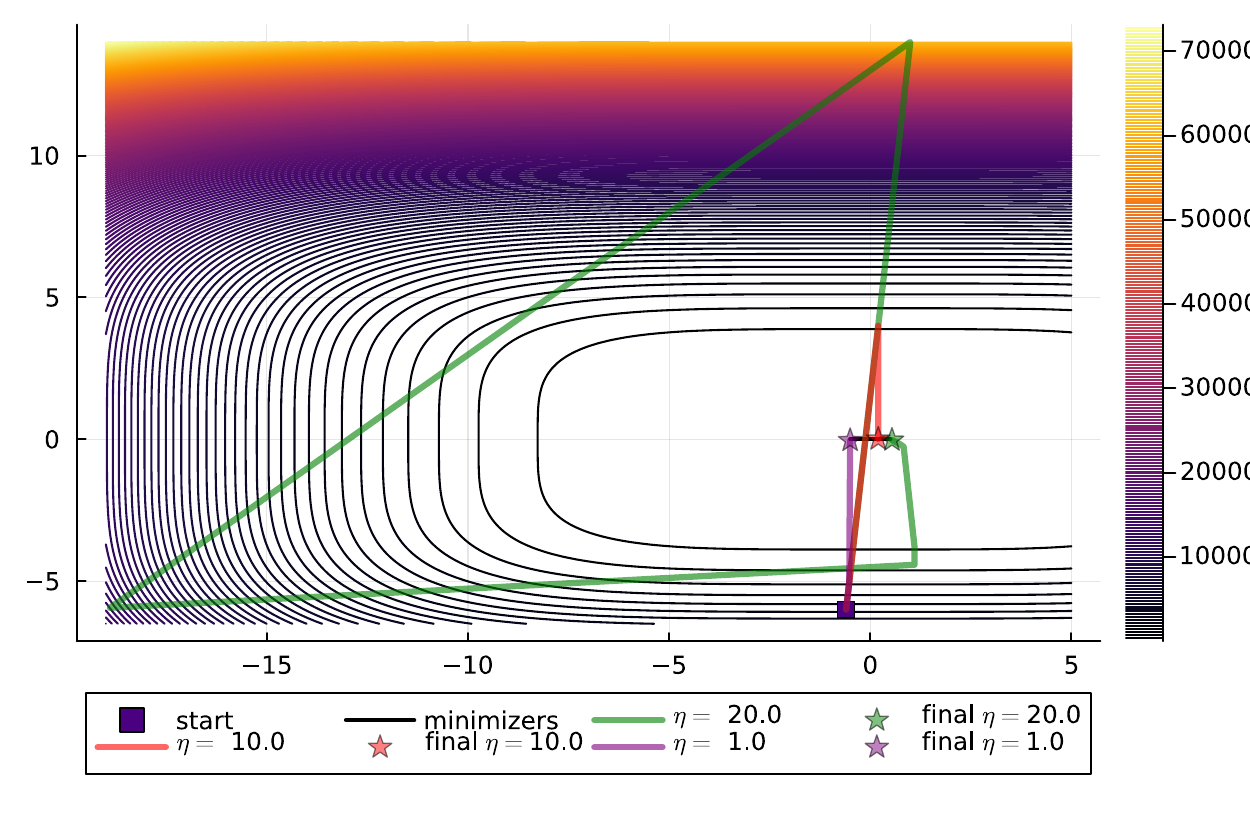}
    \caption{Trajectory of AdaDiag with different step sizes on the function $f^{\mathrm{seg}}$.}
    \label{fig:contoursegadadiag}
\end{figure}

We also observe the trajectory of AdaDiag with various step sizes on
a second test function:
\begin{align*}
f^{\cosh}(x, y) = \left( \max\left\{ |x| -\frac12, 0\right\} \right)^4 + \cosh(2 y)
\end{align*}
which is identical to $f^{\mathrm{seg}}$ in the $x$ component but with an exponential growth in the $y$ component.
The function is $(L_0, L_1)$ smooth but not $L$-smooth, nor polynomially modifiable due to the exponential growth.
The trajectories are represented in \cref{fig:contourcosh}.

\begin{figure}[ht]
    \centering
    \includegraphics[width=1.0\linewidth]{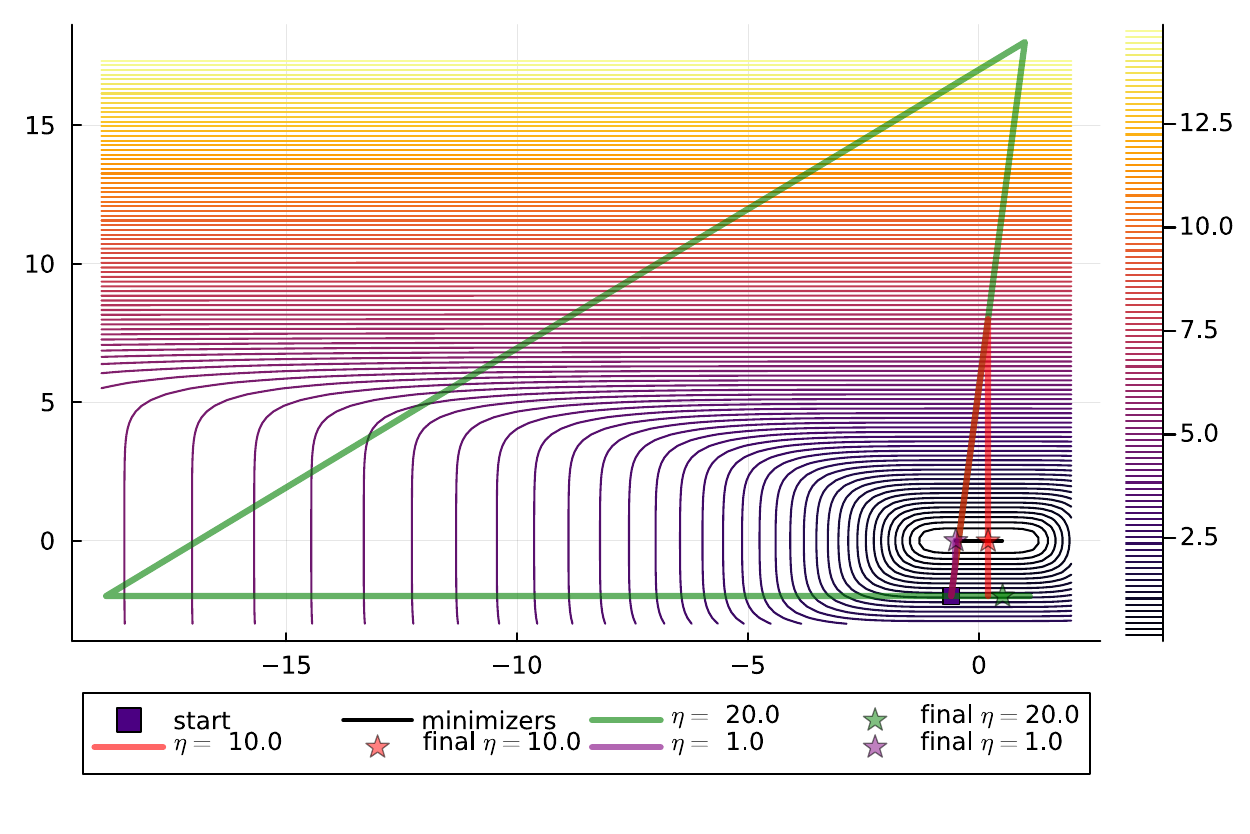}
    \caption{Trajectory of AdaDiag on the function $f^{\cosh}$. The contour levels are in log10 scale for readability.}
    \label{fig:contourcosh}
\end{figure}

Despite the algorithm being executed for $10^6$ iterations, we observe the largest step size stalling.
This is a known hurdle of AdaGrad algorithms, because the first steps reach regions with a very high gradient magnitude, the conditioning matrix $G_k^{-1/2}$
drops to very small values, in this case $10^{-15}$ in the y coordinate,
making the progress in that direction close to zero.

\section{An empirical counterexample to sequential convergence without convexity}

In this section, we provide empirical evidence that the convex assumption may be necessary for our results (e.g., \Cref{theorem:unified-framework}, \Cref{theorem:adagrad-l0-l1-modifiable-convergence}, etc.) to remain valid. 

Indeed, if $f$ is convex, differentiable and $\argmin f = \emptyset$, a valid instance is $f(x) = \exp(-x)$, which is convex, and $(0,1)$-smooth.
It is not hard to see that $x_k$ diverges to $+\infty$.
Thus, the assumption $\argmin f \neq \emptyset$ cannot be dropped.

We now consider a counterexample where $f$ is $L$-smooth and has a nonempty set of minimizers, but sequential convergence does not hold.  
The function $f: \RR^2 \to \RR $ takes inspiration from the construction of the \emph{Mexican hat function} from \citet{absil2005convergence}.
For parameters $(a, b, c, q, \omega)$, it is defined as:
\begin{align}
f(x, y) &= f_{\mathrm{rad}}( 1 - \sqrt{x^2 + y^2}, \mathrm{atan2}(y, x) ) \\
f_{\mathrm{rad}}( s, \theta ) &= 
\begin{cases}
    c s^{a+1} + q s^b (1 + \cos(\theta + \omega  \ln(s))) & \text{if } s \geq 0 \\
    0 & \text{otherwise.}
\end{cases}\nonumber
\end{align}
The function is $C^2$ on $\RR^2 \setminus \{0\}$ and has zero Hessian outside of the unit ball (provided that $a$ and $b$ are big enough).
It is not continuous at the origin where $\theta$ is undefined, which is irrelevant for the AdaGrad trajectories which stay bounded away from the origin.
The function could thus be replaced with a smooth extension around the origin to ensure it is globally $C^2$ with an upper-bounded Hessian,
making it $L$-smooth (and also $(L_0,L_1)$-generalized smooth/polynomially modifiable smooth).
In addition, it is nonnegative and attains its minimum of zero on any point with $\norm{\cdot} \geq 1$.
We run AdaGrad with parameters $(\eta=0.01, \delta=0.05)$
on the instance defined by $(a  = 5, b  = 3, c  = 1, q  = 0.3, \omega = 6)$ from a starting point
$(s_0=0.4, \theta_0 = -\omega \ln(s_0))$.
The function contour lines and the AdaGrad trajectory are displayed in \cref{fig:fulltrajectory}, clearly showing the infinite spiraling behavior of AdaGrad.
More precisely, the iterates form a sequence with an empirical rate $\theta_k \sim \ln k$,
as shown in \cref{fig:anglesequence}. Therefore, it appears impossible to have sequential convergence with smoothness alone. Other observations/results on this example are provided in the supplementary materials. \newline

We also highlight that convexity might not be the only sufficient condition for sequential convergence. Other known potential sufficient conditions include the Kurdyka-\L ojasiewicz inequality \cite{absil2005convergence,attouch2005convergence} or the Lipschitz stratification \cite{lai2026diametersubgradientsequencesominimal,chzhen2026convergenceratessubgradientdescent}; a promising future direction will be to investigate whether AdaGrad algorithms can converge in iterates under these conditions.

\begin{figure}[ht]
    \centering
    \includegraphics[width=1.1\linewidth]{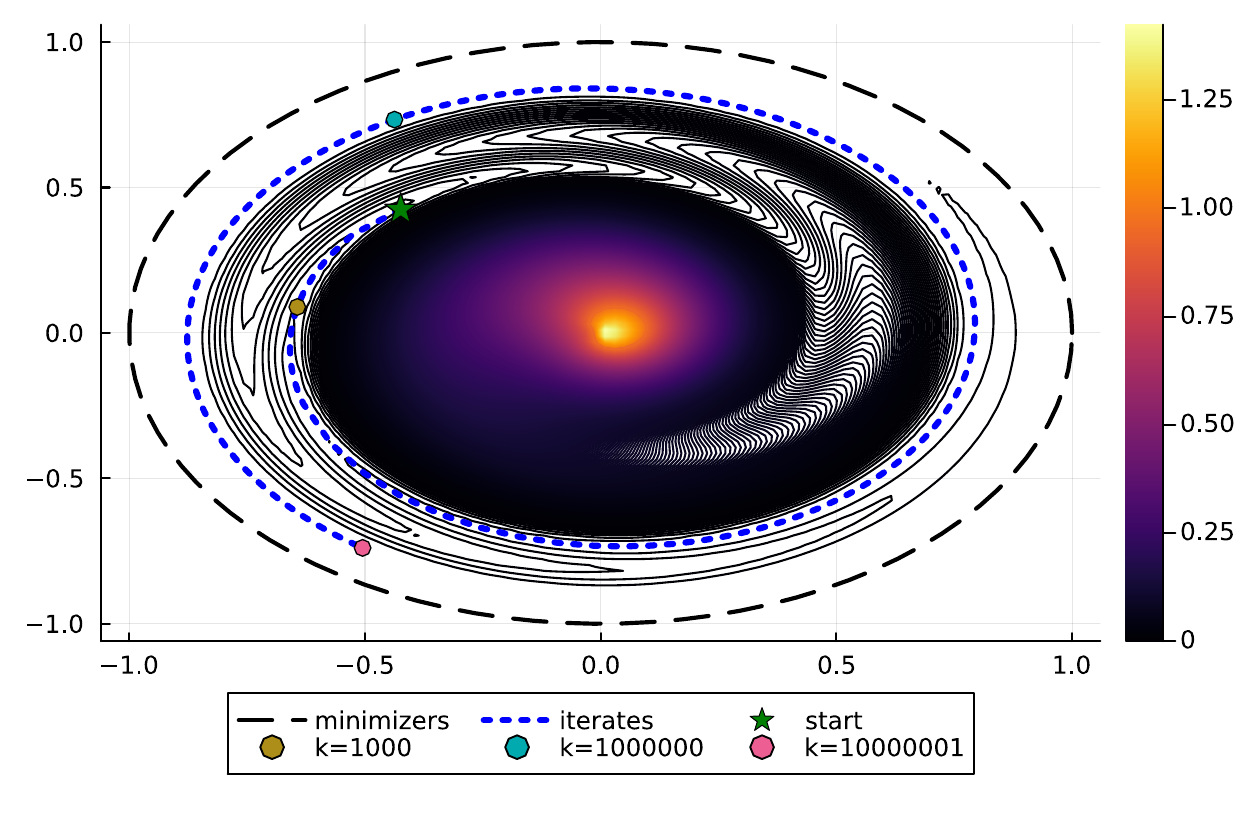}
    \caption{Trajectory of AdaGrad on the counterexample. Selected points are shown from the trajectory. Contour lines represent the function level sets.}
    \label{fig:fulltrajectory}
\end{figure}

\begin{figure}[ht]
    \centering
    \includegraphics[width=0.9\linewidth]{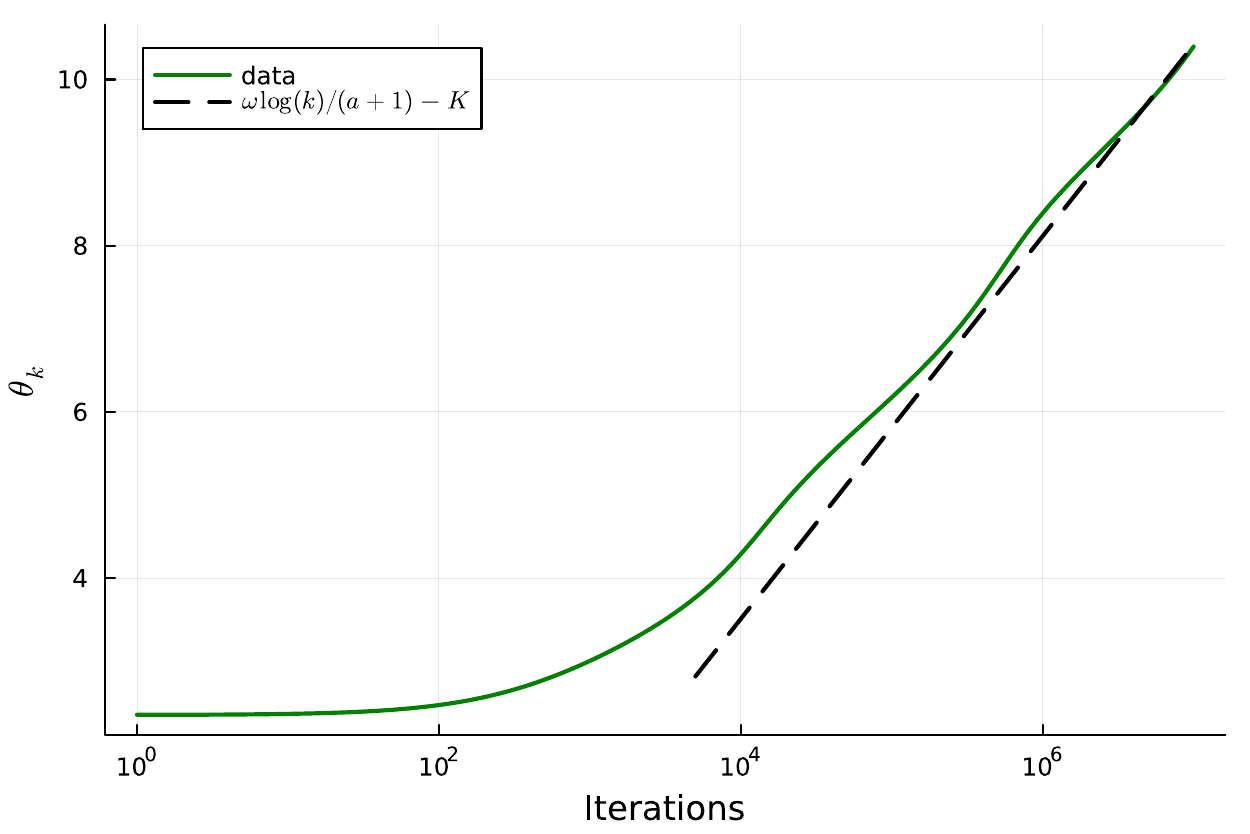}
    \caption{Sequence of iterate angles $\theta_k$ for the trajectory displayed in \cref{fig:fulltrajectory}.}
    \label{fig:anglesequence}
\end{figure}

\section{Conclusion}
In this paper, we deepened our understanding of the sequential convergence and the adaptivity of AdaGrad algorithms. In particular, we showed that three variants \ref{step:AdaNorm}, \ref{step:AdaDiag} and \ref{step:AdaFull} converge in iterates under $(L_0,L_1)$-generalized smoothness with sufficiently small constant step sizes, and under the novel $(L_0,L_1)$-polynomially modifiable smoothness with arbitrary step size using a unified framework (\Cref{theorem:unified-framework}).
We argue that our assumption on convexity is necessary since smoothness alone does not appear to be sufficient, based on a counterexample derived from the Mexican hat function.

Our work provides several results but also opens new questions.
One of them is defining other sufficient conditions for sequential convergence such as the Kurdyka-\L ojasiewicz inequality or the Lipschitz stratification. Another avenue of open questions includes the extension of these results to stochastic and constrained settings.

\section*{Acknowledgments}

This work was partially supported by the French ANR through the MIAI Cluster @ Grenoble (reference ANR-23-IACL-0006).
We would like to thank Panayotis Mertikopoulos for enlightening conversations on AdaGrad.

\bibliography{refs}

\appendix

\section{Variable metric quasi-Fejér monotonicity}

We recall the framework of \citet{combettes2013variable},
as used in \citet{traore2021sequential}. This is our main tool to prove sequential convergence.

\begin{definition}[Variable metric quasi-Fejér monotonicity]
\label{def:fejer}
Let $\{W_k\}_{k \in \N}$ be a sequence of symmetric positive definite
matrices satisfying $W_k \succeq \alpha I$ for all $k$ and some
$\alpha > 0$.
Let $C \subset \R^n$ be nonempty, closed, and convex.
A sequence $\{x_k\}_{k \in \N}$ in $\R^n$ is \emph{variable metric
quasi-Fejér monotone} with respect to $C$ relative to $\{W_k\}$ if
there exist summable nonnegative sequences $\{\rho_k\}_{k \in \N}$ and,
for each $z \in C$, $\{\varepsilon_k(z)\}_{k\in\N}$, in $\ell^1_+(\N)$,
such that
\begin{equation*}
  \normW{x_{k+1} - z}{W_{k+1}}^2
  \;\leq\;
  (1 + \rho_k)\,\normW{x_k - z}{W_k}^2 + \varepsilon_k(z),
  \, \forall\, k \in \N.
\end{equation*}
\end{definition}

In the scalar setting, $W_k = w_k I$ for a positive scalar $w_k$,
and the definition reduces to: for all $z \in C$,
\begin{equation*}
  w_{k+1}\norm{x_{k+1} - z}^2
  \;\leq\;
  (1 + \rho_k)\,w_k\norm{x_k - z}^2 + \varepsilon_k(z).
\end{equation*}

\begin{proposition}[Boundedness, {\cite[Proposition~3.2]{combettes2013variable}}]
\label{prop:bounded}
Let $\{x_k\}_{k\in\N}$ be variable metric quasi-Fejér monotone relative to a
nonempty closed convex set $C$. Then $\{x_k\}_{k\in\N}$ is bounded.
\end{proposition}

\begin{theorem}[Sequential convergence,
{\cite[Theorem~3.3]{combettes2013variable}}]
\label{thm:CV}
Let $(W_k)$ satisfy $W_k \succeq \alpha I$ for all $k$ and suppose
$W_k \to W_\infty$ as $k \to \infty$.
Let $\{x_k\}_{k\in\N}$ be variable metric quasi-Fejér monotone with respect to
a closed set $C \subset \R^n$.
Then $\{x_k\}_{k\in\N}$ converges to a point in $C$ if and only if every cluster
point of $\{x_k\}_{k\in\N}$ belongs to $C$.
\end{theorem}

%======================================================================

\section{Proofs for the sequential convergence of AdaGrad under local L-smoothness}

We denote with $\Gamma_k$ the preconditioning matrix of AdaGrad algorithms, i.e.~$\Gamma_k = G_k^{-\frac12}$.

\begin{proof}[Proof of \cref{theorem:unified-framework}]
We will prove that $1) \implies 2) \implies 3) \implies 1)$.
\begin{enumerate}[leftmargin=*]
    \item $1) \implies 2)$: this is trivial by convergence to a single minimizer.
    \item $2) \implies 3)$: Since the sequence $\{x_k\}_{k\in\NN}$ is bounded, there exists a constant $M$ such that $\|x_k\| \in B(0,M), \forall k \in \NN$. Since the gradient of $f$ is locally Lipschitz, there exists a Lipschitz constant for the gradient of $f$ restricted to the ball $B(0,M)$, i.e., one can assume that $f$ is $L$-smooth.
    For \ref{step:AdaNorm} and \ref{step:AdaDiag}, these claims were proved in \citet[Lemma 3 and 4]{traore2021sequential} respectively. For \ref{step:AdaFull}, we provide the proof as follows: by applying the descent lemma for an $L$-smooth function, i.e., $f(y) \leq f(x) + \ip{\nabla f(x)}{y - x} + \frac{L}{2}\|y - x\|^2$, we have: for all $k \in \NN$,
    \begin{equation*}
        f(x_{k+1}) \leq f(x_k) - \eta g_k^\top \Gamma_k\,g_k
       + \frac{\eta^2 L}{2}\,g_k^\top \Gamma_k^2\,g_k,
    \end{equation*}
    where $\Gamma_k = G_k^{-1/2}$ and $g_k = \nabla f(x_k)$. To bound $S_T := \sum_{k = 0}^T \|\nabla f(x_k)\|^2$, we telescope and get:
    \begin{equation}
        \label{eq:telescope-L-smooth}
        \eta\left(\sum_{k = 0}^T g_k^\top\Gamma_kg_k\right) \leq f(x_0) - f(x_{T+1}) + \frac{\eta^2  L}{2} \left(\sum_{k = 0}^T g_k^\top \Gamma_k^2 g_k\right)
    \end{equation}
    We estimate each component as follows: since 
    \begin{equation}
        \label{eq:upper-bound-G-k}
        \|G_k\| \leq \delta^2 + \sum_{j = 0}^k \|g_jg_j^\top\| \leq \delta^2 + \sum_{j = 0}^k \|g_j\|^2 = \delta^2 + S_k,
    \end{equation}
    we deduce that $\Gamma_k \succeq \frac{1}{\sqrt{\delta^2 + S_T}}I, \forall k \leq T$, we get:
    \begin{equation}
        \label{eq:estimate-LHS}
        \sum_{k=0}^T g_k^\top\Gamma_kg_k \geq \frac{1}{\sqrt{\delta^2 + S_T}} \sum_{k = 0}^T \|g_k\|^2 = \frac{S_T}{\sqrt{S_T + \delta^2}}.
    \end{equation}
    Moreover, since $G_k = G_{k-1} + g_k g_k^\top$ with $G_{-1} = \delta^2 I$, the Sherman--Morrison formula gives
    \begin{align}\label{eq:shermanmorrison}
    g_k^\top G_k^{-1}\,g_k
    = \frac{g_k^\top G_{k-1}^{-1}\,g_k}
         {1 + g_k^\top G_{k-1}^{-1}\,g_k}.        
    \end{align}
    Since $x/(1+x) \leq \ln(1+x)$ for $x \geq 0$, summing and telescoping via the matrix determinant lemma:
    \begin{align*}
    \det(G_k) = \det(G_{k-1})(1 + g_k^\top G_{k-1}^{-1} g_k),
    \end{align*}
    \begin{equation}
        \label{eq:estimate-RHS}
        \begin{aligned}
            \sum_{k=0}^{T} g_k^\top G_k^{-1}\,g_k & = \sum_{k = 0}^T \frac{g_k^\top G_{k - 1}^{-1}g_k}{1 + g_k^\top G_{k - 1}^{-1}g_k}\\
            & \leq
            \sum_{k=0}^{T} \ln(1 + g_k^\top G_{k-1}^{-1}\,g_k)\\
            &= \sum_{k = 0}^T \ln(\det(G_{k})) - \ln(\det(G_{k - 1}))\\
            &= \ln \frac{\det(G_T)}{\delta^{2n}}\\
            &\leq n\ln \frac{\delta^2 + S_T}{\delta^2} \\
            &= n \ln \left(1 + \frac{S_T}{\delta^2}\right).
        \end{aligned}
    \end{equation}
    Combining \eqref{eq:telescope-L-smooth},\eqref{eq:estimate-LHS},\eqref{eq:estimate-RHS}, we obtain:
    \begin{equation*}
        \frac{\eta S_T}{\sqrt{S_T + \delta^2}} \leq f(x_0) - f^\star + \frac{\eta^2 L}{2} n \ln\left(1 + \frac{S_T}{\delta^2}\right),
    \end{equation*}
    where $f^\star = \inf f$, we conclude that $S_T$ is finite since the LHS scales like $\sqrt{S_T}$ while the RHS scales like $\ln(S_T)$. 
    \item $3) \implies 1)$: For \ref{step:AdaNorm} and \ref{step:AdaDiag}, these implications are proved in \citet[Theorem 3.1]{traore2021sequential}. For AdaFull, it is given as in the next paragraph.

    If $S_T$ is bounded, we deduce from \Cref{eq:upper-bound-G-k} that $\{G_k\}_{k \in \NN}$ is bounded and monotone increasing (in the sense $G_{k} \succeq G_{k-1}, \forall k \geq 1$). By \citet[Theorem 1.1]{behrndt2010monotone}, $G_\infty := \lim_{k \to \infty} G_k$ exists and satisfying $\|G_\infty\| \leq \delta^2 + S_\infty$ (where $S_\infty := \lim_{k \to \infty} S_k$).   

    To prove that $\{x_k\}_{k\in\NN}$ converges to a minimizer of $f$, it is sufficient to show that:
    \begin{enumerate}[leftmargin=*]
        \item Every cluster point of $\{x_k\}$ belongs to $C := \argmin f$.
        This is true since by our hypothesis, $\{\norm{\nabla f(x_k)}^2 \}_{k\in\NN}$ is summable, which implies that the set of cluster points is included in that of critical points (and equal to $\argmin f$ due to the convexity of $f$). 
        \item The sequence $\{x_k\}_{k \in \NN}$ is variable metric quasi-Fejér monotone w.r.t. $C$ relative to the sequence $\{W_k\}_{k \in \NN}, W_k = G_{k-1}^{1/2}, \forall k \in \NN$ (with the convention that $G_{-1} = \delta^2 I$).
        By invoking the result of \citet{combettes2013variable}, we conclude the proof.
    \end{enumerate}
    For readers who are not familiar with quasi-Fejér monotonicity, we refer to the first section of this supplementary materials.
    In particular, we have to prove that there exist a summable nonnegative sequence $\{\rho_k\}_{k \in \NN}$ and for each $z \in C$, $\{\epsilon_k(z)\}_{k \in \NN} \in \ell_+^1(\NN)$ such that:
    \begin{equation*}
        \normW{x_{k+1} - z}{W_{k+1}}^2
        \leq (1 + \rho_k)\,\normW{x_k - z}{W_k}^2 + \varepsilon_k(z), \forall\, k \in \N.
    \end{equation*}
    Indeed, fix an element $z \in \argmin f$. We have:
    \begin{align*}
        x_{k+1} - z = x_k - z - \eta \Gamma_{k} g_k,
    \end{align*}
    where we recall that $\Gamma_k = G_k^{-1/2}$. Hence, $W_{k+1} \Gamma_k = I$ and $\Gamma_k W_{k+1} \Gamma_k = \Gamma_k$:
    \begin{equation}
        \label{eq:adafull-expand}
        \normW{x_{k+1}-z}{W_{k+1}}^2 = \normW{x_k-z}{W_{k+1}}^2 - 2\eta\ip{g_k}{x_k - z} + \eta^2\normW{g_k}{\Gamma_k}^2.
    \end{equation}
By convexity, $\ip{g_k}{x_k - z} \geq f(x_k) - f(z) \geq 0$:
\begin{equation}\label{eq:adafull-after-convexity}
  \normW{x_{k+1}-z}{W_{k+1}}^2
  \leq \normW{x_k-z}{W_{k+1}}^2 + \eta^2\normW{g_k}{\Gamma_k}^2.
\end{equation}
Moreover,
\begin{align*}
  \normW{u}{W_{k+1}}^2 & \leq \normW{u}{W_k}^2 + \norm{W_{k+1}-W_k}\norm{u}^2 \\
  & \leq \Bigl(1 + \frac{\norm{W_{k+1}-W_k}}{\delta}\Bigr) \normW{u}{W_k}^2,
\end{align*}
since $W_k \succeq \delta I$. Define:
\begin{align*}
& \rho_k = \norm{W_{k+1}-W_k}/\delta \\
& \varepsilon_k = \eta^2\normW{g_k}{\Gamma_k}^2.
\end{align*}
Then
\[
  \normW{x_{k+1}-z}{W_{k+1}}^2
  \leq (1+\rho_k)\normW{x_k-z}{W_k}^2 + \varepsilon_k.
\]
\emph{Summability of $\{\rho_k\}$:}
By \cref{lem:sqrt-lip} and the fact that $G_{k}, G_{k + 1} \succeq \delta^2 I$, we have:
\begin{align*}
    \norm{W_{k+1} - W_k} &= \norm{G_{k}^{1/2} - G_{k-1}^{1/2}} \\
  &\leq \frac{1}{2\delta}\norm{G_{k} - G_{k-1}} = \frac{\norm{g_{k}}^2}{2\delta},
\end{align*}
combining with the summability of $\{\|g_k\|^2\}_{k\in\NN}$, implying that $\{\rho_k\}_{k\in\NN}$ is summable.\\
\emph{Summability of $\{\varepsilon_k\}$:}
Since $\Gamma_k \preceq \delta^{-1}I$,
$\normW{g_k}{\Gamma_k}^2 \leq \norm{g_k}^2/\delta$, so
$\varepsilon_k \leq C\norm{g_k}^2$ for a finite $C$.
\end{enumerate}
\end{proof}

\begin{lemma}[Operator Lipschitz property of $t \mapsto t^{1/2}$]
\label{lem:sqrt-lip}
Let $A, B$ be Hermitian with $A \succeq B \succeq mI$ for $m > 0$.
Then
\begin{equation}\label{eq:sqrt-lip}
  \norm{A^{1/2} - B^{1/2}}
  \;\leq\; \frac{1}{2\sqrt{m}}\,\norm{A - B}.
\end{equation}
\end{lemma}

\begin{proof}
We use the resolvent integral for the matrix square root.
For $M \succ 0$,
\[
  M^{1/2}
  = \frac{1}{\pi}\int_0^{\infty} t^{-1/2}\,M\,(tI+M)^{-1}\,dt.
\]
Therefore
\[
  A^{1/2} - B^{1/2}
  = \frac{1}{\pi}\int_0^{\infty}
    t^{-1/2}\bigl[A(tI+A)^{-1} - B(tI+B)^{-1}\bigr]\,dt.
\]
Using $M(tI+M)^{-1} = I - t(tI+M)^{-1}$, the integrand simplifies:
\[
  A(tI+A)^{-1} - B(tI+B)^{-1}
  = t\bigl[(tI+B)^{-1} - (tI+A)^{-1}\bigr].
\]
Moreover, since
$(tI+B)^{-1} - (tI+A)^{-1} = (tI+B)^{-1}(A-B)(tI+A)^{-1}$,
substituting and combining give us:
\[
  A^{1/2} - B^{1/2}
  = \frac{1}{\pi}\int_0^{\infty}
    t^{1/2}(tI+B)^{-1}(A-B)(tI+A)^{-1}\,dt.
\]
Since $A \succeq B \succeq mI$,
$\norm{(tI+A)^{-1}} \leq (t+m)^{-1}$ and likewise for $B$.
Taking operator norms:
\[
  \norm{A^{1/2} - B^{1/2}}
  \leq \frac{\norm{A-B}}{\pi}
       \int_0^{\infty} \frac{t^{1/2}}{(t+m)^2}\,dt.
\]
The substitution $t = m u^2$ gives
$\int_0^{\infty} t^{1/2}(t+m)^{-2}\,dt
= (2/m^{1/2})\int_0^{\infty} u^2(u^2+1)^{-2}\,du
= \pi/(2m^{1/2})$.
Therefore
$\norm{A^{1/2}-B^{1/2}} \leq \norm{A-B}/(2\sqrt{m})$.
\end{proof}

\section{Sequential convergence under $(L_0, L_1)$-smoothness}

In this section, we provide the proofs for the convergence results of AdaGrad algorithms under $(L_0, L_1)$-smoothness.

\begin{proof}[Proof of \cref{theorem:adagrad-l0-l1-convergence}]  
    We start with the proof for \ref{step:AdaDiag}. The treatment of \ref{step:AdaFull} will be given later.
    \paragraph{Proof for \ref{step:AdaDiag}} 
    Leveraging \Cref{theorem:unified-framework}, it is sufficient to prove that for any compact $\mathcal{C} \subseteq \RR^n$ and $\delta > 0$, there exists $\eta^\star$ such that if $\eta \leq \eta^\star$,
    $\{\|\nabla f(x_k)\|_2^2\}_{k\in\NN}$ is summable, i.e.,
    \begin{equation*}
        \sum_{k \in \NN} \|\nabla f(x_k)\|_2^2 < +\infty.
    \end{equation*}
    W.l.o.g., we assume that $f^\star := \inf f = 0$.

    Indeed, by the descent lemma for $(L_0,L_1)$-smooth functions, we obtain that:
    \begin{equation*}
        \begin{aligned}
            &\bigl|f(y) - f(x) - \ip{\nabla f(x)}{y-x}\bigr| \leq \\ &\frac{L_0 + L_1\norm{\nabla f(x)}}{2} \,\|y - x\|^2 \,\exp\bigl(L_1\norm{y-x}\bigr).
        \end{aligned}
    \end{equation*}
    Considering a learning rate threshold $\eta^\star < \frac{\ln 2}{L_1\sqrt{n}}$, one can obtain that:
    \begin{equation*}
        L_1\|x_{k+1}-x_k\| \leq \eta L_1\sqrt{n} \leq \ln 2.
    \end{equation*}
    As a consequence, we get:
    \begin{equation}
        \label{eq:l0-l1-descent-lemma-adagrad}
        \begin{aligned}
            &f(x_{k + 1}) - f(x_k) \\ 
            &\leq \ip{g_k}{x_{k + 1}-x_k}
            + (L_0 + L_1\norm{g_k}) \,\|x_{k+1} - x_k\|^2\\
            &= -\eta \underbrace{g_k^\top G_k^{-1/2} g_k}_{=: \Psi_k} + \eta^2 (L_0 + L_1\norm{g_k}) \underbrace{g_k^\top G_k^{-1}g_k}_{=: \Phi_k}.
        \end{aligned}
    \end{equation}
    Telescoping the above inequalities from $k = 0$ to $T$ gives us:
    \begin{equation}
        \label{eq:telescope-version-1}
        \eta \sum_{k = 0}^T \Psi_k \leq {f(x_0) - f^\star} + \eta^2 \sum_{k = 0}^T (L_0 + L_1\|g_k\|) \Phi_k.
    \end{equation}
    Let $\Delta(\mathcal{C}):= \max_{x_0 \in \mathcal{C}} f(x_0) - f^\star$, we deduce that with an initialization $x_0 \in \mathcal{C}$, one has:
    \begin{equation*}
        \eta \sum_{k = 0}^T \Psi_k \leq \Delta(\mathcal{C}) + \eta^2 \sum_{k = 0}^T (L_0 + L_1\|g_k\|) \Phi_k.
    \end{equation*}
    In the next step, we will provide an estimation for the summation in both sides of the above inequality. Recalling that $S_{T} = \sum_{k = 0}^T \|g_k\|^2$, we obtain:
    \begin{equation*}
        \begin{aligned}
            \sum_{k=0}^T \Psi_k &= \sum_{k=0}^T g_k^\top G_k^{-1/2}g_k \geq \sum_{k = 0}^T \frac{\|g_k\|^2}{\sqrt{\delta^2 + S_T}} = \frac{S_T}{\sqrt{S_T + \delta^2}}.
        \end{aligned}
    \end{equation*}
    Moving to the RHS terms, denote $V_{k,i} = \sum_{j = 0}^k g_{j,i}^2$, we have:
    \begin{equation*}
        \begin{aligned}
            \sum_{k = 0}^T \Phi_k &= \sum_{k = 0}^T g_k^\top G_k^{-1} g_k = \sum_{i = 1}^n \sum_{k = 0}^T \frac{g_{k,i}^2}{{\delta^2 + V_{k,i}}}\\
            &\leq \sum_{i = 1}^n \ln\left(1 + \frac{V_{T,i}}{\delta^2}\right) \leq n\ln \left(1 + \frac{S_T}{\delta^2}\right).
        \end{aligned}
    \end{equation*}
    Moreover,
    \begin{equation*}
        \begin{aligned}
            \sum_{k=0}^T \|g_k\|\Phi_k &= \sum_{i = 1}^n \sum_{k = 0}^T \|g_k\| \frac{g_{k,i}^2}{\delta^2 + V_{k,i}} \\
            &\leq \sum_{i = 1}^n \sqrt{\sum_{k = 0}^T \frac{g_{k,i}^2}{\delta^2 + V_{k,i}}} \sqrt{\sum_{k = 0}^T \|g_k\|^2\frac{g_{k,i}^2}{\delta^2 + V_{k,i}}}\\
            &\leq \sum_{i = 1}^n \sqrt{\ln\left(1 + \frac{V_{T,i}}{\delta^2}\right)}\sqrt{\sum_{k=0}^T \|g_k\|^2}\\
            &\leq \sum_{i = 1}^n \sqrt{\ln\left(1 + \frac{S_{T}}{\delta^2}\right)S_T}\\
            &= n \sqrt{\ln\left(1 + \frac{S_T}{\delta^2}\right)S_T}.
        \end{aligned}
    \end{equation*}
    Combining these estimations, we obtain:
    \begin{equation}
        \label{eq:key-inequality}
        \begin{aligned}
            \frac{S_T}{\sqrt{S_T + \delta^2}} \leq \frac{\Delta(\mathcal{C})}{\eta} + &\eta L_0 n \ln\left(1 + \frac{S_T}{\delta^2}\right) + \\
            &\eta L_1 n \sqrt{\ln\left(1 + \frac{S_T}{\delta^2}\right)S_T}.
        \end{aligned}
    \end{equation}
    Note that for a given $\eta > 0$, \Cref{eq:key-inequality} is not enough to prove the finiteness of $S_T$ since the LHS scales like $O(\sqrt{S_T})$ whereas the RHS scales like $O(\sqrt{S_T\ln{S_T}})$ and thus, no contradiction to be found if $S_T \overset{T \to \infty}{\to} +\infty$. Nevertheless, if $\eta$ is chosen small enough, one can prove that there exists an interval $[s_1, s_2], s_2 \geq s_1 > 0$ satisfying:
    \begin{enumerate}[leftmargin=*]
    \item $s_1 \geq s_{\min}:= 10(B(\cC)+ \delta^2 / 3) + 2\frac{L_0^2}{L_1^2}$
    with the constant $B(\cC):= \max_{x \in \mathcal{C}} \|\nabla f(x)\|^2$.
    \item $s_2 > 10 s_1 + \frac{2L_0^2}{L_1^2}$.
    \item $h(x) > 0, \forall x \in [s_1,s_2]$.
    \end{enumerate}
    where:
    \begin{multline*}
        h(x) = \frac{\sqrt{x}}{2} - \frac{\Delta(\mathcal{C})}{\eta} \\ -\eta n \left(L_0\ln\left(1 + \frac{x}{\delta^2}\right) + L_1\sqrt{x\ln\left(1 + \frac{x}{\delta^2}\right)}\right).
    \end{multline*}
    Before justifying the existence of $s_1,s_2$ (with a proper choice of learning rate threshold $\eta^\star$), we explain why it is sufficient to finish the proof: one can bound $S_T \leq s_1, \forall T \in \NN$. Indeed, the key relation that we exploit is the following: if $f$ is $L_0,L_1$-smooth, by \citet[Lemma 2.5]{vankov2025optimizing}:
    \begin{equation}
        \label{eq:gradient-relation}
        \|\nabla f(y) - \nabla f(x)\| \leq (L_0 + L_1\|\nabla f(x)\|)\frac{\exp(L_1\|x-y\|) - 1}{L_1}.
    \end{equation}
    If we apply \Cref{eq:gradient-relation} with $y = x_{k + 1}, x = x_{k}$ and $\eta \leq \frac{\ln 2}{L_1\sqrt{n}}$, we get:
    \begin{equation*}
        \|g_{k + 1}\| \leq 2\|g_k\| + \frac{L_0}{L_1}.
    \end{equation*}
    Using the fact that $(a + b)^2 \leq 2a^2 + 2b^2$, we have:
    \begin{equation*}
        \|g_{k + 1}\|^2 \leq 8\|g_k\|^2 + 2\frac{L_0^2}{L_1^2} \leq 8S_k + 2\frac{L_0^2}{L_1^2}, \forall k \in \NN.
    \end{equation*}
    Therefore,
    \begin{equation}
        \label{eq:next-G-k-estimation}
        S_{k + 1} = S_k + \|g_{k+1}\|^2 \leq 9S_k + 2\frac{L_0^2}{L_1^2}, \forall k \in \NN.
    \end{equation}
    we consider two cases:
    \begin{enumerate}[leftmargin=*]
        \item If $S_T \leq \delta^2/3, \forall T \in \NN$, then $S_T \leq \delta^2 \leq s_1$ by our assumption on $s_1$.
        \item Otherwise, $S_k \geq \delta^2 / 3$ starting from a certain index $K \in \NN$. Therefore:
        \begin{equation}
            \label{eq:estimation-lhs}
            \frac{S_k}{\sqrt{S_k + \delta^2}} \geq \frac{1}{2} \sqrt{S_k}, \forall k \geq K.
        \end{equation}
        We also argue that $S_K \leq s_1$ because: if $K = 0$, $S_0 = \|g_0\|^2 \leq s_1$ by our assumption on $s_1$. Otherwise, $S_{K-1} \leq \delta^2/3$ and by \eqref{eq:next-G-k-estimation}:
        \begin{equation*}
            S_K \leq 9 S_{K - 1} + 2\frac{L_0^2}{L_1^2} \leq 9 \frac{\delta^2}{3} + 2\frac{L_0^2}{L_1^2} \leq s_1,
        \end{equation*}
        again, by our choice of $s_1$. Using \eqref{eq:estimation-lhs} and \eqref{eq:key-inequality}, we conclude that $S_K \leq s_1$ and:
        \begin{equation*}
            h(S_k) \leq 0, \forall k \geq K.
        \end{equation*}
        For the sake of contradiction, assume that the increasing sequence $\{S_k\}_{k \in \NN}$ is eventually bigger than $s_1$. Let $\ell > K$ be the first index satisfying $S_\ell > s_1$. By \eqref{eq:next-G-k-estimation}, we have:
        \begin{equation*}
            S_\ell \leq 9S_{\ell - 1} + 2\frac{L_0^2}{L_1^2} \leq 9s_1 + 2\frac{L_0^2}{L_1^2} < s_2,
        \end{equation*}
        by our condition on $s_2$. Thus, $S_\ell \in [s_1,s_2]$, which implies $h(S_\ell) > 0$ - a contradiction.
    \end{enumerate}
    To finish the proof, for a fixed learning rate $\eta$, we consider two constants $s_1$ and $s_2$ as follows:
    \begin{equation}
        \label{eq:s1-s2-construction-adadiag}
        \begin{aligned}
            s_1 &= \left(\frac{8 \Delta(\mathcal{C})}{\eta}\right)^2 + \left(\frac{(32n\eta L_0)^4}{\delta^2} + \delta^2\right) + s_{\min},\\
            s_2 &= \delta^2 \left(\exp\left(\frac{1}{(4\eta L_1 n)^2}\right) - 1\right).\\
        \end{aligned}
    \end{equation}
    Note that this construction of $(s_1,s_2)$ does not automatically satisfy all our conditions on $s_1$ and $s_2$. In the rest of the proof, we argue that there exists a learning rate threshold $\eta^\star > 0$ such that for all $0 < \eta \leq \eta^\star$, such conditions will hold. 

    \begin{enumerate}[leftmargin=*]
        \item By definition, we have: $s_1 \geq s_{\min}$.
        \item Since $s_2$ is exponential w.r.t. $1/\eta$ and $s_1$ is only polynomial w.r.t. $1/\eta$ when $\eta \to 0$, there must exist a sufficiently small threshold $\eta^\star > 0$ such that if $\eta < \eta^\star$, one has: $s_2 > 10 s_1 + \frac{2L_0^2}{L_1^2}$.
        \item It remains to prove that $h(x) > 0$ for $x \in [s_1, s_2]$. Indeed, our choice of $s_2$ implies that:
        \begin{equation*}
            \eta L_1 n \sqrt{\ln\left(1 + \frac{x}{\delta^2}\right)} \leq \frac{1}{4}, \forall x \leq s_2.
        \end{equation*} 
        Our choice of $s_1$, on the other hand, implies that for $x \geq s_1$, one has:
        \begin{equation*}
            \begin{aligned}
                \frac{1}{8}\sqrt{x} &> \frac{\Delta(\mathcal{C})}{\eta},\\
                \frac{1}{8}\sqrt{x} &> \frac{1}{8}\sqrt[4]{x - \delta^2}\sqrt[4]{x + \delta^2} \geq \frac{4n\eta L_0}{\sqrt{\delta}}\sqrt[4]{x + \delta^2}\\
                &= 4n\eta L_0 \sqrt[4]{1 + \frac{x}{\delta^2}}\\
                &\geq 4n\eta L_0 \ln \left(\sqrt[4]{1 + \frac{x}{\delta^2}}\right) \quad (\text{since } x \geq \ln x, \forall x \geq 1)\\
                &= n\eta L_0 \ln \left(1 + \frac{x}{\delta^2}\right).
            \end{aligned}
        \end{equation*}
        Combining these inequalities, we achieve that for any $x \in [s_1, s_2]$:
        \begin{equation*}
            \begin{aligned}
                h(x) &= \left(\frac{1}{2} - \eta n L_1 \sqrt{\ln\left(1 + \frac{x}{\delta^2}\right)}\right)\sqrt{x} - \frac{\Delta(\mathcal{C})}{\eta}\\ 
                &- \eta n L_0 \ln\left(1 + \frac{x}{\delta^2}\right)\\
                &\geq \frac{1}{4} \sqrt{x} - \frac{\Delta(\mathcal{C})}{\eta} - \eta n L_0 \ln\left(1 + \frac{x}{\delta^2}\right) > 0. \\
            \end{aligned}
        \end{equation*}
    \end{enumerate}
    This concludes the proof for \ref{step:AdaDiag}.

    \paragraph{Proof for \ref{step:AdaFull}}
    We sketch the main difference between the proof of \ref{step:AdaDiag} and \ref{step:AdaFull}.
    In the case of \ref{step:AdaFull},
    since $\norm{x_{k+1} - x_k}^2 = \eta^2 \Phi_k$ and
    $\Phi_k \leq 1$ by \eqref{eq:shermanmorrison},
    we only need
    $\eta \leq \frac{\ln 2}{L_1}$ to obtain the generalized descent \eqref{eq:l0-l1-descent-lemma-adagrad} and the similar telescope \Cref{eq:telescope-version-1}. From here, the only difference is the estimation of $\sum_k \|g_k\|\Phi_k$.
    Recalling that $\Phi_k \leq 1$ by the Sherman-Morrison formula from \eqref{eq:shermanmorrison}, 
    we have:
    \begin{equation*}
        \begin{aligned}
            \sum_{k=0}^T \|g_k\|\Phi_k &= \sum_{k = 0}^T \|g_k\| g_k^\top G_k^{-1}g_k \\
            &\leq \sqrt{\sum_{k = 0}^T g_k^\top G_k^{-1}g_k} \sqrt{\sum_{k = 0}^T \|g_k\|^2g_k^\top G_k^{-1}g_k}\\
            &\leq \sqrt{n\ln\left(1 + \frac{S_T}{\delta^2}\right)}\sqrt{\sum_{k=0}^T \|g_k\|^2}\\
            &= \sqrt{n\ln\left(1 + \frac{S_T}{\delta^2}\right)S_T}.
        \end{aligned}
    \end{equation*}
    Therefore, \emph{the only difference is the dependence $\sqrt{n}$ instead of $n$ as in \ref{step:AdaDiag}}.
    In the end, we have the following key inequality:
    \begin{equation}
        \label{eq:key-inequality-adafull}
        \begin{aligned}
            \frac{S_T}{\sqrt{S_T + \delta^2}} \leq \frac{\Delta(\mathcal{C})}{\eta} + &\eta L_0 n \ln\left(1 + \frac{S_T}{\delta^2}\right) + \\
            &\eta L_1 \sqrt{n\ln\left(1 + \frac{S_T}{\delta^2}\right)S_T}.
        \end{aligned}
    \end{equation}
    The rest is fairly similar: the barriers $s_1,s_2$ are defined as:
    \begin{equation}
        \label{eq:s1-s2-construction-adafull}
        \begin{aligned}
            s_1 &= \left(\frac{8 \Delta(\mathcal{C})}{\eta}\right)^2 + \left(\frac{(32n\eta L_0)^4}{\delta^2} + \delta^2\right) + s_{\min},\\
            s_2 &= \delta^2 \left(\exp\left(\frac{1}{(4\eta L_1)^2n}\right) - 1\right).\\
        \end{aligned}
    \end{equation}
    Arguing similarly to the case of \ref{step:AdaDiag}, we get the final result.
\end{proof}

\section{Explicit estimation of $\eta^\star$ in \Cref{theorem:adagrad-l0-l1-convergence}}

In this section, we give an explicit estimation of $\eta^\star$
for \cref{theorem:adagrad-l0-l1-convergence}.
It is necessary to solve the inequalities $s_2 \geq 10s_1 + 2\frac{L_0^2}{L_1^2}$ with $s_1, s_2$ given as in \Cref{eq:s1-s2-construction-adadiag} and \Cref{eq:s1-s2-construction-adafull}.
We need the following technical lemma:

\begin{lemma}[Estimation of logarithm inequalities]
    \label{lemma:technical-estimation}
    Consider three scalars $a, b, c > 0$. If $x \geq \frac{\ln (1 + 2c + 16(b+1)^2/a^2)}{a}$, then:
    \begin{equation*}
        \exp(ax) \geq bx + c.
    \end{equation*}
\end{lemma}
\begin{proof}
    The inequalities can be rewritten as:
    \begin{equation*}
        \exp(ax) - c \geq bx.
    \end{equation*}
    Since $x \geq \frac{\ln(1 + 2c + 16(b+1)^2/a^2)}{a}$, we conclude that:
    \begin{equation*}
        \begin{aligned}
            \exp(ax) &\geq 2c + 1 + \frac{16(b+1)^2}{a^2}> 2c,\\
            \exp\left(\frac{ax}{2}\right) &\geq \frac{ax}{2} + 1 \geq \frac{ax}{2},\\
            \exp\left(\frac{ax}{2}\right) &\geq \sqrt{2c + 1 + \frac{16(b+1)^2}{a^2}} \geq \frac{4b}{a}.
        \end{aligned}
    \end{equation*}
    Therefore,
    \begin{equation*}
    \exp(ax) - c \geq \frac{\exp(ax)}{2} =  \frac{\exp\left(\frac{ax}{2}\right)\exp\left(\frac{ax}{2}\right)}{2} \geq \frac{4axb}{4a} = bx.
    \end{equation*}
\end{proof}
Using \Cref{lemma:technical-estimation}, we obtain the following bound on $\eta^\star$:
\begin{enumerate}[leftmargin=*]
    \item For \ref{step:AdaDiag}, let $x = \frac{1}{\eta^2}$, we need to bound the inequality $\exp(ax) \geq bx + c$ with:
    \begin{equation*}
        \begin{aligned}
            a &= \frac{1}{(4nL_1)^2},\\
            b &= \frac{640\Delta(\cC)^2}{\delta^2},\\
            c &= \frac{10(32\sqrt{n} L_0)^4}{(L_1\delta)^4} + \frac{100 B(\cC)}{\delta^2} + \frac{22L_0^2}{\delta^2L_1^2} + 45,
        \end{aligned}
    \end{equation*}
    where we use the requirement $\eta \leq \frac{\ln 2}{L_1\sqrt{n}} \leq \frac{1}{L_1\sqrt{n}}$ to remove the term $\eta^2$ in $s_1$ (cf. \Cref{eq:s1-s2-construction-adadiag}).
    Applying \Cref{lemma:technical-estimation}, we get:
    \begin{equation*}
        \frac{1}{\eta^2} \geq (4nL_1)^2 \ln (1 + 2c + 16(b+1)^2/a^2),
    \end{equation*}
    which implies that:
    \begin{equation*}
        \eta \leq \frac{1}{4nL_1\sqrt{\ln(1 + 2c + 16(b+1)^2/a^2)}}.
    \end{equation*}
    Combining with the constraint: $\eta \leq \frac{\ln 2}{L_1\sqrt{n}}$, we conclude that it is sufficient to take:
    \begin{equation*}
        \eta^\star = \frac{1}{4nL_1 \sqrt{\ln(1 + 2c + 16(b+1)^2/a^2)}},
    \end{equation*}
    since $c \geq 45$.
    \item For \ref{step:AdaFull}, we use the same machinery, which is to consider $b$ similar to that of \ref{step:AdaDiag} and:
    \begin{equation*}
        \begin{aligned}
            a &= \frac{1}{(4L_1)^2 n}\\
            c &= \frac{10(32{n} L_0)^4}{(L_1\delta)^4} + \frac{100 B(\cC)}{\delta^2} + \frac{22L_0^2}{\delta^2L_1^2} + 45, 
        \end{aligned}
    \end{equation*}
    where we use the requirement $\eta \leq \frac{\ln 2}{L_1} \leq \frac{1}{L_1}$ to remove the term $\eta^2$ in $s_1$ (cf. \Cref{eq:s1-s2-construction-adafull}).
    With the same argument, we get:
    \begin{equation*}
        \eta^\star = \frac{1}{4\sqrt{n}L_1\sqrt{\ln(1 + 2c + 16(b+1)^2/a^2)}},
    \end{equation*}
    since $c \geq 45$.
\end{enumerate}

\section{Numerical experiments on the counterexample}\label{sec:numerics}

In this section, we present additional results on the counterexample.
We first briefly detail the computational setup.
All experiments are run in Julia 1.12.6 \citep{bezanson2014julia} on a Linux laptop equipped with an AMD Ryzen 7 7730U CPU and 16GB of RAM.
Crucial to the fact that this example does not exhibit sequential convergence,
\cref{fig:gradnorm} displays the sequence of gradient norms.
We highlight in particular its non-summability -- $\sum_k \norm{g_k}$ is asymptotically larger than  $\sum_{k} 1/k$, a necessary condition for non-convergence of iterates.
We also display the sequence of function values in \cref{fig:funcval},
highlighting that unlike the $(L_0,L_1)$ counterexample from \citet{bohan2023ada}, the method still converges in function value,
since the function is L-smooth for some constant $L$.

\begin{figure}
    \centering
    \includegraphics[width=0.9\linewidth]{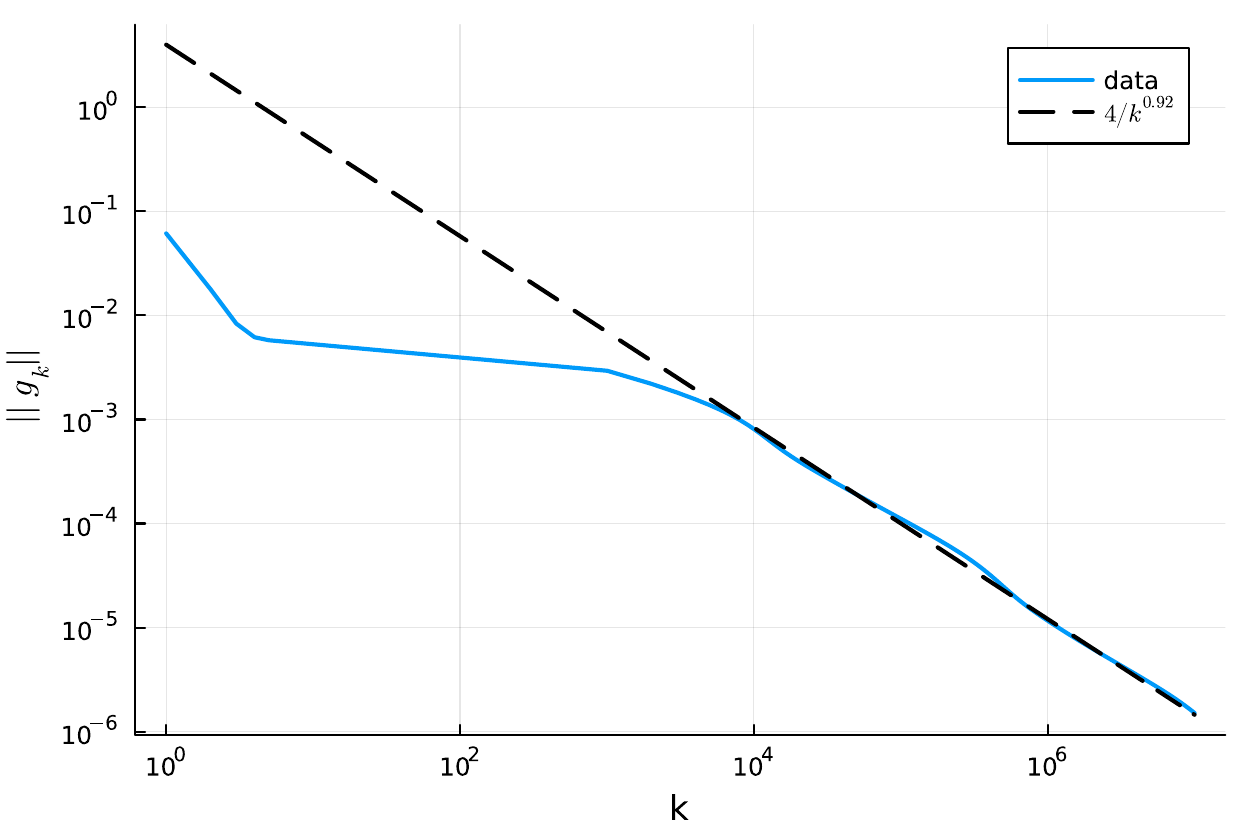}
    \caption{Sequence of gradient norms produced by AdaGrad on the example.}
    \label{fig:gradnorm}
\end{figure}

\begin{figure}
    \centering
    \includegraphics[width=0.9\linewidth]{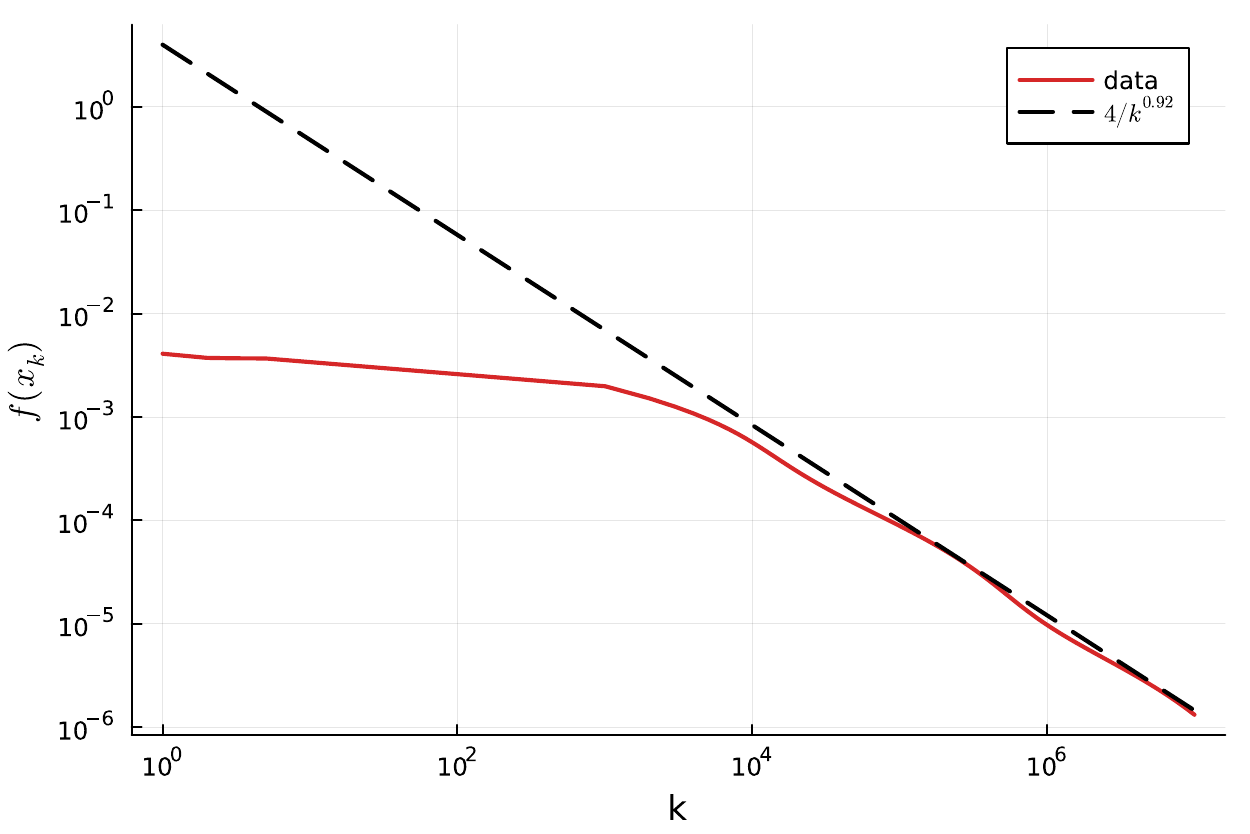}
    \caption{Sequence of function values produced by AdaGrad on the example.}
    \label{fig:funcval}
\end{figure}

\begin{figure}
    \centering
    \includegraphics[width=0.9\linewidth]{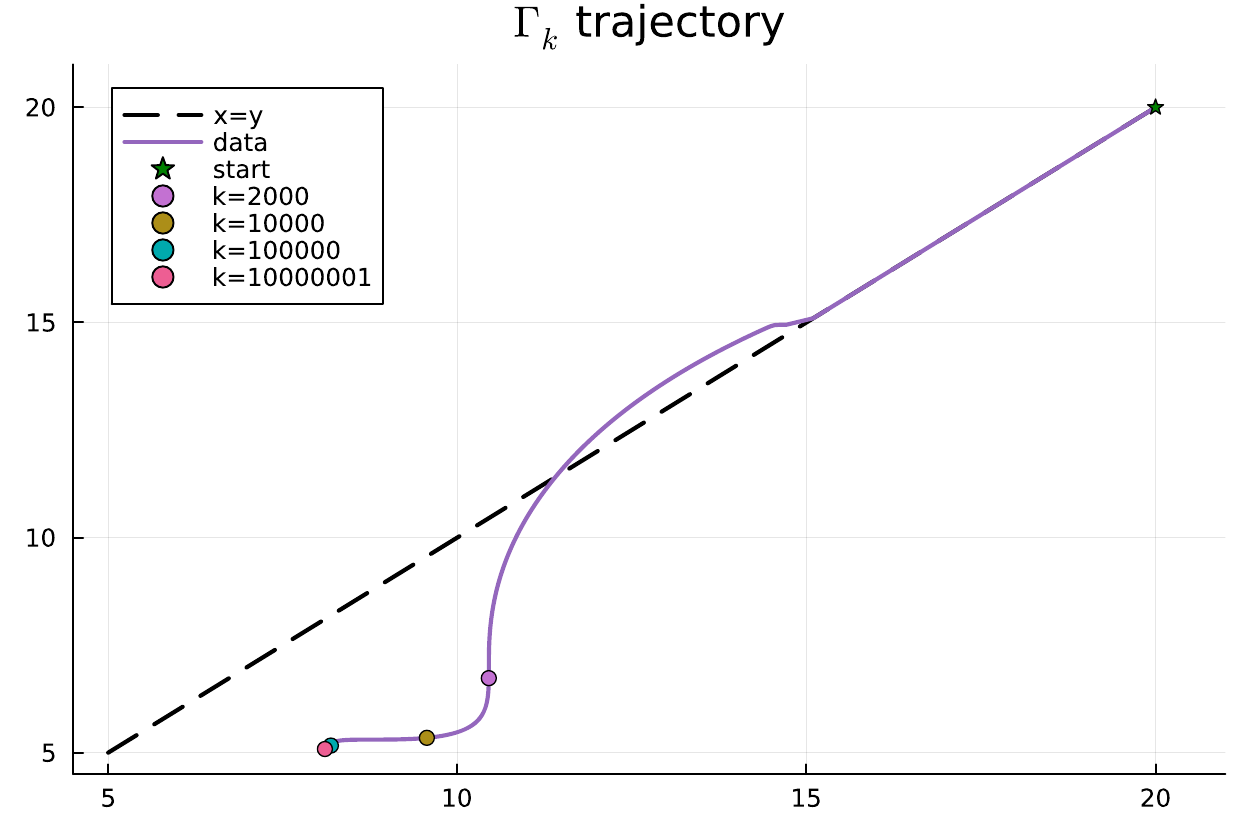}
    \caption{Trajectory of the preconditioning matrix $\Gamma_k$.}
    \label{fig:gammatraj}
\end{figure}

\section{Spiral behavior of AdaGrad}

In this section, we show that running AdaGrad algorithms on a variant of the Mexican hat function might result in a spiral behavior: the iterates themselves are bounded, with monotonically decreasing function values, but their set of accumulation points is not a singleton.
We point out that the function is slightly different from the one used in the numerical experiments: we changed the sign from $+$ to $-$ in front of the $\cos$ function and $\omega = 1$. This is an adjustment that we made after the main text submission since it makes the equations in the subsequent proof become much cleaner. The proof for the function in the main text is identical.  

\subsection{Definition of the ``pathological'' function}

We first expose the function construction from the counter-example in more details.
Consider the following function $\tilde{f}: \RR^2 \setminus B\left(0,\frac{1}{2}\right) \to \RR$, given in polar coordinates as:

\begin{equation}
    \label{eq:mexican-hat-function-variant}
    \tilde{f}(r,\theta) = \begin{cases}
        (1 - r)^{a+1} + (1 - r)^{b} (1 - \cos(\theta + \ln (1 - r))), &\text{if } 1/2 \leq r < 1\\
        0, &\text{otherwise}
    \end{cases} 
\end{equation}
where the two integers $a, b \in \NN$ satisfy: $b \geq 3, a > b - 1$. This is a variant of the Mexican hat function from \citep{absil2005convergence}.

To ease the presentation, we will also work with the variable $s = 1 - r$ rather than $r$, which gives the function:
\begin{equation}
    \label{eq:mexican-hat-function-variant-variable-s}
    f(s,\theta) = \begin{cases}
        s^{a+1} + s^{b} (1 - \cos(\theta + \ln s)), &\text{if } 1/2 \leq r < 1\,(\text{or equivalently } 0 < s \leq 1/2),\\
        0, &\text{otherwise}
    \end{cases} 
\end{equation}

Note that we have to exclude a ball around the origin because, under the current form, the function $\tilde{f}$ is not well-defined at $(0,0)$.
Nevertheless, in the following, we will prove that $\tilde{f}$ is $C^2$ on $\RR^2 \setminus B\left(0, \frac{1}{2}\right)$.
This allows us to use the Whitney extension theorem \citep{whitney1934analytic} to extend $\tilde{f}$ to a $C^2$ function $f:\RR^2 \to \RR$ that agrees with $\tilde{f}$ on $\RR^2 \setminus B\left(0, \frac{1}{2}\right)$.

We first provide a property of the function \eqref{eq:mexican-hat-function-variant}.
\begin{lemma}[Smoothness of $f$]
    \label{lemma:property-of-f}
    The function $\tilde{f}$ is $C^2$ on $\RR^2 \setminus B\left(0,\frac{1}{2}\right)$.
\end{lemma}
\begin{proof}
    Denote $x = r\cos(\theta), y = r\sin(\theta)$. We define two vectors:
    \begin{equation*}
        \hat{r} = \begin{pmatrix}
            \cos\theta \\ \sin\theta
        \end{pmatrix},\qquad 
        \hat{\theta} = \begin{pmatrix}
            -\sin\theta \\ \cos\theta
        \end{pmatrix}
    \end{equation*}
    The gradient of $\tilde{f}$ (or $f$) in Cartesian coordinates outside of the unit disk is $0$, and that in $B\left(0,1\right) \setminus B\left(0, \frac{1}{2}\right)$ are given by:
    \begin{equation*}
        \begin{aligned}
            \nabla f(x) = 
                (\nabla_r f) \hat{r} + (\nabla_\theta f) \frac{1}{r} \hat{\theta},
        \end{aligned}
    \end{equation*}
    where
    \begin{equation}
        \label{eq:gradient-f-polar}
        \begin{aligned}
            \nabla_r f &= -(a+1)(1 - r)^{a} - b(1-r)^{b-1}(1 - \cos(\theta + \ln(1 - r))) - (1 - r)^{b - 1}\sin(\theta + \ln(1 - r)),\\
            \nabla_\theta f &= (1 - r)^b \sin(\theta + \ln(1-r)).
        \end{aligned}
    \end{equation}
    Since both $\nabla_r f$ and $\nabla_\theta f$ tend to zero when $r \to 1$ (since $b \geq 3$ and $a > b - 1$), we conclude that $f$ is $C^1$.
    
    Similarly, we compute the Hessian of $\tilde{f}$ in Cartesian coordinates.
    We focus on $x \in B(0,1) \setminus B\left(0, \frac{1}{2}\right)$ since otherwise, the Hessian is simply $0$. It is given by:
    \begin{equation*}
        \nabla^2 f(x) = \begin{pmatrix}
            \hat{r} & \hat{\theta}
        \end{pmatrix} \begin{pmatrix}
            \nabla_{rr} f & \frac{1}{r} \nabla_{r\theta} f - \frac{1}{r^2} \nabla_\theta f\\
            \frac{1}{r} \nabla_{r\theta} f - \frac{1}{r^2} \nabla_\theta f & \frac{1}{r} \nabla_r f + \frac{1}{r^2}\nabla_{\theta\theta} f
        \end{pmatrix}\begin{pmatrix}
            \hat{r} & \hat{\theta}
        \end{pmatrix}^\top,\forall x \in B(0,1) \setminus B\left(0, \frac{1}{2}\right),
    \end{equation*}
    where:
    \begin{equation*}
        \begin{aligned}
            \nabla_{rr} f = \nabla_{rr}^2 f &= a(a+1)(1-r)^{a - 1} + b(b-1)(1-r)^{b - 2}(1 - \cos(\theta + \ln(1-r)))\\ 
            &+ b(1-r)^{b-2}\sin(\theta + \ln(1-r)) + (b-1)(1-r)^{b-2}\sin(\theta + \ln(1 - r))\\
            &+ (1-r)^{b-2}\cos(\theta + \ln(1-r)),\\
            \nabla_{\theta r} f = \nabla_{r\theta} f &= -b(1-r)^{b-1}\sin(\theta + \ln(1-r)) - (1-r)^{b-1}\cos(\theta + \ln(1-r)),\\
            \nabla_{\theta\theta} f &= (1-r)^b \cos(\theta + \ln(1-r)).
        \end{aligned}
    \end{equation*}
    One can verify that with $b \geq 3, a > b - 1$, $\nabla_{rr} f, \nabla_{\theta r} f, \nabla_{\theta\theta} f$ tend to zero when $r \to 1$, hence $\nabla^2 f$. Thus, the function $\tilde{f}$ is $C^2$. 
\end{proof}

A direct corollary of \Cref{lemma:property-of-f} is:
\begin{corollary}[L-smoothness extension]\label{cor:L-smoothness}
    There exists an $L$-smooth function $f: \RR^2 \to \RR$ that agrees with $\tilde{f}$ on its domain, i.e., $\RR^2 \setminus B\left(0,\frac{1}{2}\right)$. 
\end{corollary}
\begin{proof}
    Via the Whitney extension theorem \citep{whitney1934analytic} and \Cref{lemma:property-of-f}, there exists a $C^2$ function $f$ agreeing with $\tilde{f}$ on its domain. Moreover, the Hessian is bounded since $\nabla^2 f (x) = 0, \forall \|x\| > 1$. Thus, the function is also $L$-smooth.
\end{proof}

From now on, we will refer to $f$ as the function defined in \eqref{eq:mexican-hat-function-variant} (and also \eqref{eq:mexican-hat-function-variant-variable-s}). Essentially, we do not care what the extension $f$ looks like around the origin since with a well-chosen initialization and step-size, the iterates of AdaGrad never reach $B\left(0, \frac{1}{2}\right)$. Hence, the dynamics of the AdaGrad algorithms is only governed by $\tilde{f}$ (or simply $f$).

\subsection{Formal statement on the spiral behavior of AdaGrad algorithms}

\begin{theorem}[Spiral behavior of AdaGrad algorithms]
    \label{theorem:sprial-behavior}
    Consider the $L$-smooth function $f$ given as in \Cref{cor:L-smoothness}. There exists an initialization $x_0$, a step size $\eta$ and a constant $\delta > 0$ such that the sequence of polar coordinates $\{(r_k, \theta_k)\}_{k\in\NN}$ corresponding to that of the iterates $\{x_k\}_{k \in \NN}$ satisfies:
    \begin{equation*}
        \begin{aligned}
            \forall k \in \NN, r_k &< 1, \quad \text{ or equivalently } \quad s_k = 1 - r_k > 0,\\
            \lim_{k \to \infty} r_k &= 1, \quad \text{ or equivalently } \quad \lim_{k \to \infty} s_k = 0,\\
            \lim_{k \to \infty} \theta_k + \ln(1 - r_k) &= 0,\\
            \lim_{k \to \infty} |\theta_k - \theta_{k+1}| &= 0.
        \end{aligned}
    \end{equation*}
\end{theorem}

The last equality implies that $\theta_k \approx -\ln (1 - r_k)$. The angle $\theta_k$ in our analysis is thus a continuous variable taking values in $\RR$: it is an accumulated continuous angle along the whole sequence.
More precisely, we choose a fixed value of $\theta_0 \in \RR$ such that $x_0 = (r_0 \cos(\theta_0), r_0 \sin(\theta_0))$. We then define $\theta_k$ recursively by
\begin{equation*}
    \theta_{k+1} = \theta_k + \Delta_k,
\end{equation*}
where $\Delta_k \in (-\pi, \pi]$ represents the directed angle shift from $x_k$ to $x_{k+1}$. Specifically, $\Delta_k$ is positive if the transition from $x_k$ to $x_{k+1}$ is counterclockwise, and negative if it is clockwise.

Before proving the results, we explain why these fourth elements explain the spiraling behavior: 
\begin{enumerate}[leftmargin=*]
    \item The first two equations: $r_k < 1, \forall k \in \NN$ and $\lim_{k \to \infty} r_k = 1$ imply that the iterates eventually approach the unit circle but never escape the unit disk.
    \item The last equation implies that $\theta_k$ grows infinitely because $- \ln(1 - r_k) \to \infty$ (since $\lim_{k \to \infty} r_k = 1$). But this is not enough to show the spiraling behavior (e.g., take $\theta_k = 2\pi k$). Therefore, we need the last limit: $\lim_{k \to \infty} |\theta_k - \theta_{k + 1}| = 0$ to avoid these possible ``jumps'' that allow the iterates to converge and ensure that the iteration trajectory spirals, and hence, converge to a non-singleton accumulation set.  
\end{enumerate}

To prove \Cref{theorem:sprial-behavior}, we use the generic template of AdaGrad given by:
\begin{equation}
    \label{eq:abstract-adagrad}
    x_{k + 1} = x_{k} - \eta \Gamma_k \nabla f(x_k).
\end{equation}
where $\Gamma_k$ is the preconditioning matrix defined by the formulation of three variants of the AdaGrad algorithms.
We want to control $\eta \Gamma_k$, which is performed through the following lemma: 
\begin{lemma}[Choice of step-size $\eta$ and the constant $\delta$]
    \label{lemma:step-size-choice}
    If we choose $\eta, \delta > 0$:
    \begin{equation*}
        \begin{aligned}
            \frac{\eta}{\delta} \leq \frac{1}{L},    
        \end{aligned} 
    \end{equation*}
    where $L$ is the smoothness constant of the function $f$ (given in \Cref{cor:L-smoothness}), then we have:
    \begin{equation*}
        \frac{\eta}{\delta + (f(x_0) - f^\star) / \eta} I\preceq \eta \Gamma_k \preceq \frac{1}{L}I,
    \end{equation*}
    where $f^\star = \inf_{x \in \RR^2} f(x) > -\infty$. In particular, with $\eta = 1$ and $\delta \geq L$, we have:
    \begin{equation*}
        \frac{1}{\delta + f(x_0) - f^\star} I \preceq \eta \Gamma_k \preceq \frac{1}{\delta} I.
    \end{equation*}
\end{lemma}
\begin{proof}
    Since the function $f$ is $L$-smooth, we have:
    \begin{equation*}
        f(y) \leq f(x) + \innp{\nabla f(x)}{y - x} + \frac{L}{2} \|y - x\|^2.
    \end{equation*}
    By the choice of $\eta$ and $\delta$, we know that $\eta \Gamma_k \preceq \frac{1}{L}I$, therefore:
    \begin{equation}
        \label{eq:descent-phenomena-large-delta}
        \begin{aligned}
            f(x_{k + 1}) &\leq f(x_k) - \nabla f(x_k)^\top \eta \Gamma_k\nabla f(x_k) + \frac{L}{2} \eta^2 \nabla f(x_k)^\top \Gamma_k^2 \nabla f(x_k)\\
            &\leq f(x_k) - \nabla f(x_k)^\top \left(\eta \Gamma_k - \frac{L}{2} \eta^2 \Gamma_k^2  \right)\nabla f(x_k) \\
            &\leq f(x_k) - \frac{1}{2} \eta \nabla f(x_k)^\top \Gamma_k \nabla f(x_k).
        \end{aligned}
    \end{equation}
    Telescoping gives us:
    \begin{equation*}
        \sum_{k = 0}^T \eta \nabla f(x_k)^\top \Gamma_k \nabla f(x_k) \leq f(x_0) - f^\star.
    \end{equation*}
    Note that we have the same estimation for the three AdaGrad variants:
    \begin{equation*}
        \sum_{k = 0}^T \eta \nabla f(x_k)^\top \Gamma_k \nabla f(x_k) \geq \eta\frac{S_T}{\sqrt{S_T + \delta^2}},
    \end{equation*}
    where $S_T = \sum_{k = 0}^T \|\nabla f(x_k)\|^2$. Therefore,
    \begin{equation}
        \label{eq:sum-square-grad-bounded}
        \eta\frac{S_T}{\sqrt{S_T + \delta^2}} \leq f(x_0) - f(x^\star) \implies \sqrt{S_T + \delta^2} \leq \frac{f(x_0) - f^\star}{\eta} + \delta,
    \end{equation}
    which implies that:
    \begin{equation*}
        S_T \leq \left(\frac{f(x_0) - f^\star}{\eta} + \delta\right)^2 - \delta^2
    \end{equation*}
    Finally, we have: 
    \begin{equation*}
        \eta \Gamma_k \succeq \frac{\eta}{\sqrt{\delta^2 + S_T}} I \succeq \frac{\eta}{\delta + (f(x_0) - f^\star) / \eta} I.
    \end{equation*}
    For our particular function, it is sufficient to fix $\eta = 1$.
\end{proof}

We now consider the update of $s_k$ (or $1 - r_k$) and $\theta_k$ with the three AdaGrad algorithms. 

\begin{lemma}[Dynamics of AdaGrad algorithms]
    \label{lemma:dynamics-adagrad-algos}
    Consider AdaGrad algorithms with $\eta = 1$. Let the polar coordinates of the $k$-th iteration $x_k$ be $(r_k, \theta_k)$ and define: 
    \begin{equation*}
        \begin{aligned}
            \hat{r}_k &= \begin{pmatrix}
                \cos\theta_k \\ \sin\theta_k
            \end{pmatrix}, \\
            \hat{\theta}_k &= \begin{pmatrix}
                -\sin\theta_k \\ \cos\theta_k
            \end{pmatrix},\\
            \beta_k &= \theta_k + \ln s_k.
        \end{aligned}
    \end{equation*}
    We define in addition:
    \begin{equation*}
        \begin{aligned}
            \rho_{k}^r &:= \innp{\hat{r}_k}{\Gamma_k\nabla f(x_k)},\\
            \rho_{k}^\theta &:= \innp{\hat{\theta}_k}{\Gamma_k\nabla f(x_k)}.\\
        \end{aligned}
    \end{equation*}
    If $\rho_k^r < r_k$, we obtain the following dynamics: 
    \begin{equation}
        \label{eq:dynamics-s-beta}
        \begin{aligned}
            s_{k + 1} &= 1 - \sqrt{(r_k - \rho_k^r)^2 + (\rho_k^\theta)^2},\\
            \theta_{k+1} &= \theta_k - \arctan \frac{\rho_k^\theta}{r_k - \rho^r_k},\\
            \beta_{k + 1} &= \beta_k - \arctan \frac{\rho_k^\theta}{r_k - \rho^r_k} + \ln \frac{s_{k+1}}{s_{k}},\\
            \|x_{k + 1} - x_k\| &= \sqrt{(\rho_k^r)^2 + (\rho_k^\theta)^2}.
        \end{aligned}
    \end{equation}
\end{lemma}
\begin{proof}
    The result is obtained by simply applying a change of coordinates: if one switches the coordinate system to use $(\hat{r}_k, \hat{\theta}_k)$ as the basis, then the current coordinates of $x_k$ are $(r_k, 0)$. Therefore,
    \begin{equation*}
        x_{k+1} = (r_k - \rho^r_k) \hat{r}_k - \rho_k^\theta \hat{\theta}_k.
    \end{equation*}
    The new $s_{k+1}$ is given by:
    \begin{equation*}
        s_{k + 1} = 1 - r_{k + 1}, \quad r_{k + 1} = \sqrt{(r_k - \rho_k^r)^2 + (\rho_k^\theta)^2}.
    \end{equation*}
    Moreover, if $\rho_k^r < r_k$, the new $\theta_{k+1}$ and $\beta_{k+1}$ can be calculated as:
    \begin{equation*}
        \theta_{k + 1} = \theta_k - \arctan \frac{\rho_k^\theta}{r_k - \rho^r_k},
    \end{equation*}
    and
    \begin{equation*}
        \begin{aligned}
            \beta_{k + 1} &= \theta_{k + 1} + \ln s_{k + 1} \\
            &= \beta_k + (\theta_{k + 1} - \theta_k) + \ln \frac{s_{k + 1}}{s_k}\\
            &= \beta_k - \arctan \frac{\rho_k^\theta}{r_k - \rho^r_k} + \ln \frac{s_{k + 1}}{s_k}.
        \end{aligned}
    \end{equation*}
    Finally, the distance $\|x_{k+1} - x_k\| = \|\Gamma_k \nabla f(x_k)\| = \sqrt{(\rho_k^r)^2 + (\rho_k^\theta)^2}$, since $(\rho_k^r, \rho_k^\theta)$ are the coordinates of $\rho_k$ in the new basis $(\hat{r}_k, \hat{\theta}_k)$.
\end{proof}

Next, we prove the first two claims in \Cref{theorem:sprial-behavior}
\begin{lemma}[Trapped in the unit ball and approaching the unit circle]
    \label{lemma:trapping-approaching}
    Consider a sequence $\{x_k\}_{k \in \NN}$ generated by AdaGrad algorithms with the hyperparameters $\eta = 1, \delta > 0$ and an initialization $x_0$ such that:
    \begin{equation*}
        \theta_0 = - \ln s_0 = - \ln (1 - r_0), 
    \end{equation*}
    where $(r_0, \theta_0)$ is the polar coordinates of $x_0$.
    Then, for a sufficiently small $s_0$ and a sufficiently large $\delta$, we have:
    \begin{equation*}
        \begin{aligned}
             \forall k \in \NN, \quad s_k &> 0,\\
            \lim_{k \to \infty} s_k &= 0.\\
        \end{aligned}
    \end{equation*}
\end{lemma}
\begin{proof}
    We first prove that $0 < s_k \leq s_0, \forall k \in \NN$ with a well-chosen $s_0$ and $\delta$.
    Note that for a fixed radius $r$, the function $f$ is minimized at:
    \begin{equation*}
        \theta = - \ln (1 - r) = - \ln s.
    \end{equation*}
    and the minimum value is:
    \begin{equation*}
        \min_{x, \|x\| = r} f(x) = (1 - r)^{a + 1},
    \end{equation*}
    which is monotonically decreasing w.r.t.~$r$. 
    We choose $x_0$ such that $r_0 \geq 3/4$ and $\delta$ is sufficiently large in the sense that:
    \begin{equation*}
        \begin{aligned}
            \delta & \geq L, \\
            \|\nabla f(x)\| & \leq \frac{\delta}{4}, \forall \|x\| \geq \frac{3}{4}.
        \end{aligned}
    \end{equation*}
    where $L$ is the smoothness constant of $f$. 

    We first prove $s_k \leq s_0, \forall k \in \NN$. Due to \eqref{eq:descent-phenomena-large-delta}, there is a descent phenomenon, i.e.,
    \begin{equation*}
        f(x_{k+1}) \leq f(x_k) \leq \ldots \leq f(x_0) 
    \end{equation*}
    Since $f(x_0) = \min_{\|x\| = r_0} f(x)$, and this function is monotonically decreasing, $r_k \notin [0.5,r_0), \forall k \in \NN$. Moreover, $r_k$ cannot be smaller than $0.5$ either since if $k$ is the first index satisfying $r_k < 0.5$, we have:
    \begin{equation*}
        \left|\frac{1}{2} - r_0\right| < \|x_{k+1}\| - \|x_k\| \leq \|x_{k + 1} - x_k\| = \|\Gamma_k\nabla f(x_k)\| \leq \frac{1}{\delta} \|\nabla f(x_k)\| \leq 1/4,
    \end{equation*}
    which is a contradiction since $r_0 \geq 3/4$. Hence, $s_k \leq s_0, \forall k \in \NN$.
    The gradient of $f$ decomposes as 
    \begin{equation*}
        \nabla f(x) = (\nabla_r f) \hat{r} + \frac{1}{r} (\nabla_\theta f) \hat{\theta},
    \end{equation*}
    where $\nabla_r f, \nabla_\theta f$ are given as in \eqref{eq:gradient-f-polar} and $\hat{r}, \hat{\theta}$ are given as in \Cref{lemma:property-of-f}.
    We established that $\|\nabla f(x_k)\| = O(s_k^2)$ (since $b \geq 3, a > b - 1$),
    so we also have $\|\Gamma_k \nabla f(x_k)\| = O(s_k^2)$. Therefore:
    \begin{equation*}
        |s_{k + 1} - s_k| = |r_{k + 1} - r_k| \overset{(\star)}{\leq} \|x_{k+1} - x_k\| = \|\Gamma_k\nabla f(x_k)\| = O(s_k^{2}),
    \end{equation*}
    where $(\star)$ is due to the triangle inequality. Thus, for sufficiently small $s_0$, we have: $\|\Gamma_k\nabla f(x_k)\| \leq \frac{1}{2}s_k$, hence:
    \begin{equation}
        s_{k + 1} \geq s_k - \|\Gamma_k\nabla f(x_k)\| \geq \frac{1}{2} s_k > 0.
    \end{equation}
    Thus, we can conclude that $0 < s_k \leq s_0, \forall k \in \NN$.
    
    Thanks to the previous claim, we deduce that the sequence $\{x_k\}_{k \in \NN}$ is bounded and so is its set of accumulation points. Due to \eqref{eq:descent-phenomena-large-delta}, we know that $\{\|\nabla f(x_k)\|^2\}_{k \in \NN}$ is summable.
    It implies that for any accumulation point $\bar{x}$ of $\{x_k\}_{k \in \NN}$ and its polar coordinates $(\bar r, \bar \theta)$:
    \begin{equation*}
        \nabla f(\bar{x}) = 0 \implies \begin{cases}
            \nabla_r f(\bar{x}) &= \,0,\\
            \nabla_\theta f(\bar{x}) &= \,0
        \end{cases}.
    \end{equation*}
    Note that if $\|\bar{x}\| < 1$, $\nabla_\theta f(\bar{x}) = 0$ only if $\sin(\bar{\theta} + \ln(1 - \bar{r})) = 0$. However, this would imply that:
    \begin{equation*}
        (a+1)(1 - \bar{r})^{a} + b(1-\bar{r})^{b-1}(1 - \cos(\bar{\theta} + \ln(1 - \bar{r}))) = 0,
    \end{equation*}
which is impossible. Hence, any accumulation point $\bar{x}$ has to lie on the unit circle, which implies that $\lim_{k \to \infty} s_k = 0$. That concludes the proof. 
\end{proof}

We finish the proof of \Cref{theorem:sprial-behavior} by the following lemma, which explains the dynamics of the angle $\theta_k$, and also $\beta_k = \theta_k + \ln(1 - r_k)$.
\begin{lemma}[Dynamics of the angle]
    \label{lemma:dynamics-angle}
    Under the same assumption as in \Cref{lemma:trapping-approaching}, for sufficiently large $\delta$, there exists a constant $C$ such that for sufficiently small $s_0$, :
    \begin{equation*}
        |\beta_k| \leq C s_k^{a - b + 1}, \forall k \in \NN.
    \end{equation*}
    In particular, $\lim_{k \to \infty} \beta_k = 0$. Moreover, $\lim_{k \to \infty} |\theta_k - \theta_{k+1}| = 0$.
\end{lemma}
\begin{proof}
    To understand the dynamics of $\beta$ using $s$, we consider the gradient of $f$ reformulated using $\beta$ and $s$ (instead of $\theta$ and $r$).
    \begin{equation*}
        \begin{aligned}
            &\nabla f(x) = \nabla_r f(x) \hat{r} + \frac{1}{r} \nabla_\theta f(x) \hat{\theta}\\
            =& [-(a+1)(1 - r)^{a} - b(1-r)^{b-1}(1 - \cos(\theta + \ln(1 - r))) - (1 - r)^{b - 1}\sin(\theta + \ln(1 - r))] \hat{r} \\
            +& \frac{1}{r}(1 - r)^b \sin(\theta + \ln(1-r)) \hat{\theta}\\
            =& \left[-(a+1)s^{a} - bs^{b-1}(1 - \cos(\theta + \ln s)) - s^{b - 1}\sin(\theta + \ln s)\right] \hat{r} + \frac{1}{r}\left[s_k^b \sin(\theta + \ln s)\right] \hat{\theta}\\
            =& \left[-(a+1)s^{a} - bs^{b-1}(1 - \cos\beta) - s^{b - 1}\sin(\beta)\right] \hat{r} + \frac{1}{r}\left[s^b \sin(\beta)\right] \hat{\theta}\\
            =&\underbrace{\left[-(a+1)s^{a} - bs^{b-1}(1 - \cos\beta)\right]}_{P} \hat{r} - \left[s^b \sin(\beta)\right]\underbrace{\left(\frac{1}{s}\hat{r} - \frac{1}{r}\hat{\theta}\right)}_{=:\zeta}.
        \end{aligned}
    \end{equation*}
    In the following, we will use $P_k$ and $\zeta_k$ to indicate those quantities at the $k$th iteration. 
    
    Next, we will linearize the dynamics of $\beta_k$. By \Cref{lemma:dynamics-adagrad-algos}, we have:
    \begin{equation*}
        \beta_{k + 1} = \beta_k - \arctan \frac{\rho_k^\theta}{r_k - \rho^r_k} + \ln \frac{s_{k+1}}{s_{k}},
    \end{equation*}
    where $\rho_k^r$ and $\rho_k^\theta$ are defined as in \Cref{lemma:dynamics-adagrad-algos}. We consider each term separately: If $s_0$ is sufficiently small, we have:
    \begin{equation*}
        \begin{aligned}
            \left|\arctan \frac{\rho_k^\theta}{r_k - \rho^r_k} - \frac{\rho_k^\theta}{r_k}\right| & \leq \left|\arctan \frac{\rho_k^\theta}{r_k - \rho^r_k} - \frac{\rho_k^\theta}{r_k - \rho^r_k}\right| + \left|\frac{\rho_k^\theta}{r_k - \rho^r_k} - \frac{\rho_k^\theta}{r_k}\right|\\
            &\leq \underbrace{\frac{1}{3}\left|\frac{\rho_k^\theta}{r_k - \rho^r_k}\right|^3}_{O(\|\nabla f(x_k)\|^3)} + \underbrace{\left|\frac{\rho_k^\theta\rho_k^r}{r_k(r_k - \rho^r_k)}\right|}_{O(\|\nabla f(x_k)\|^2)}.
        \end{aligned}
    \end{equation*}
    where we use the fact that $|\arctan x - x| \leq \frac{1}{3}|x|^3, \forall |x| < 1$ (using the Taylor expansion of the function $\arctan$ around $0$) and $|\rho_k^r|, |\rho_k^\theta| \leq \|\Gamma_k \nabla f(x_k)\| = O(\|\nabla f(x_k)\|)$.

    We deal with the second term (while always assuming that $s_0$ is sufficiently small) as:
    \begin{equation*}
        \begin{aligned}
            \left|\ln \frac{s_{k + 1}}{s_k} - \frac{\rho_k^r}{s_k} \right|&\leq \left|\ln \frac{s_{k + 1}}{s_k} - \frac{s_{k + 1} - s_k}{s_k} \right| + \left|\frac{s_{k + 1} - s_k - \rho^r_k}{s_k}\right|\\
            &= \left|\ln \frac{s_{k + 1}}{s_k} - \frac{s_{k + 1} - s_k}{s_k} \right| + \left|\frac{r_{k + 1} - r_k + \rho^r_k}{s_k}\right|\\
            &\leq \left(\frac{s_{k + 1} - s_k}{s_k}\right)^2 + \left|\frac{\sqrt{(r_k - \rho_k^r)^2 + (\rho_k^\theta)^2} - (r_k - \rho^r_k)}{s_k}\right|\\
            &\leq \underbrace{\left(\frac{s_{k + 1} - s_k}{s_k}\right)^2}_{=O\left(\frac{\|\nabla f(x)\|^2}{s_k^2}\right)} + \underbrace{\left|\frac{(\rho_k^\theta)^2}{s_k(r_k - \rho^r_k)\sqrt{(r_k - \rho_k^r)^2 + (\rho_k^\theta)^2}}\right|}_{=O\left(\frac{\|\nabla f(x)\|^2}{s_k}\right)},\\
        \end{aligned}
    \end{equation*}
    where we use again the fact that $|\ln(1 + x) - x| \leq x^2$ for $|x| < 1/2$ (using the Taylor expansion again). We conclude that:
    \begin{equation*}
        \beta_{k + 1} - \beta_k = -\frac{\rho_k^\theta}{r_k} + \frac{\rho_k^r}{s_k} + O\left(\frac{\|\nabla f(x_k)\|^2}{s_k^2}\right).
    \end{equation*}
    Due to the definitions of $\rho_k^\theta$ and $\rho_k^r$, we have:
    \begin{equation*}
        -\frac{\rho_k^\theta}{r_k} + \frac{\rho_k^r}{s_k} = -\frac{\innp{\hat{\theta}}{\Gamma_k\nabla f(x_k)}}{r_k} + \frac{\innp{\hat{r}}{\Gamma_k\nabla f(x_k)}}{s_k} = \innp{\zeta_k}{\Gamma_k\nabla f(x_k)}.
    \end{equation*}
    Thus, we end up with the following estimation:
    \begin{equation*}
        \begin{aligned}
            \beta_{k + 1} &= \beta_k + P_k\innp{\zeta_k}{\Gamma_k\hat{r}_k} - \sin \beta_k \underbrace{s^b\innp{\zeta_k}{\Gamma_k\zeta_k}}_{\Lambda_k} + O\left(\frac{\|\nabla f(x_k)\|^2}{s_k^2}\right)\\
            &=\beta_k (1 - \Lambda_k) + P_k\innp{\zeta_k}{\Gamma_k\hat{r}_k} - (\sin \beta_k - \beta_k)\Lambda_k + O\left(\frac{\|\nabla f(x_k)\|^2}{s_k^2}\right).
        \end{aligned}
    \end{equation*}
    We estimate the terms in the above equation: since $\lambda_- I \preceq \Gamma_k \preceq \lambda_+ I$ (with $\lambda_\pm$ determined uniquely by $\delta$ and $f(x_0) - f^\star$, see \Cref{lemma:step-size-choice}), $\zeta_k = \Theta(1/s_k)$, we have:
    \begin{equation*}
        \Lambda_k = s_k^b \innp{\zeta_k}{\Gamma_k\zeta_k} = \Theta(s_k^{b - 2}).
    \end{equation*}
    Moreover,
    \begin{equation*}
        \begin{aligned}
            |P_k\innp{\zeta_k}{\Gamma_k\hat{r}_k}| &= (a+1)|s_k^{a} \innp{\zeta_k}{\Gamma_k\hat{r}_k}| + b|s_k^{b-1}(1 - \cos\beta_k) \innp{\zeta_k}{\Gamma_k\hat{r}}| \\
            &\leq \underbrace{(a+1)|s_k^{a} \innp{\zeta_k}{\Gamma_k\hat{r}_k}|}_{= O(s_k^{a - 1})} + \underbrace{\frac{b}{2}|s_k^{b-1}\beta_k^2 \innp{\zeta_k}{\Gamma_k\hat{r}_k}|}_{=O(s_k^{b - 1}\beta_k^2)}
        \end{aligned} 
    \end{equation*}
    \begin{equation*}
        \begin{aligned}
            |(\sin(\beta_k) - \beta_k)\Lambda_k| \leq \frac{1}{6}|\beta_k^3\Lambda_k| = O(\beta_k^3s_k^{b - 2}).
        \end{aligned}
    \end{equation*}
    We also have to estimate the difference between $s_{k + 1}$ and $s_k$, which is:
    \begin{equation*}
        \begin{aligned}
            \|\underbrace{s_{k + 1} - s_k}_{\Delta_k}\| &\leq \|\Gamma_k\nabla f(x_k)\| = O(\|P_k \hat{r}_k - s_k^b \sin (\beta_k) \zeta_k\|) = O(s_k^a + s_k^{b - 2}\beta_k^2 + s_k^{b - 1}\beta_k).
        \end{aligned}
    \end{equation*}
    This estimation allows us to have:
    \begin{equation*}
        \begin{aligned}
            s_{k+1}^{a - b + 1} - s_k^{a - b + 1} &= O(s_k^{a - b}\|\Delta_k\|).
        \end{aligned}
    \end{equation*}
    Thus, we come to the following estimation: We prove by induction that $|\beta_k| \leq Cs_k^{a - b + 1}$ for a big enough constant $C > 0$ (that we will explain how to choose later). For $k = 0$, the claim holds trivially since $\beta_0 = 0$. Assume that the claim holds up to $k \in \NN$, we have the following order for each quantity that we just estimated under the assumption that $a > b - 1$:
    \begin{equation*}
        \begin{aligned}
            \Lambda_k &= \Theta(s_k^{b - 2}),\\
            |P_k\innp{\zeta_k}{\Gamma_k\hat{r}_k}| &= O(s_k^{a - 1} + s_k^{b-1 + 2(a-b+1)}) = O(s_k^{a - 1}),\\
            |(\sin(\beta_k) - \beta_k)\Gamma_k| &= O(\beta_k^3s_k^{b - 2}) = O(s_k^{3(a - b + 1) + b - 2}) = O(s_k^{3a -2b + 1}),\\
            \|s_{k+1}-s_k\| &= O(s_k^{a} + s_k^{b - 1 + 2(a - b + 1)} + s_k^{b - 1 + a - b + 1}) = O(s_k^{a}),\\
            s_{k + 1}^{a - b + 1} - s_k^{a - b + 1} &= O(s_k^{2a - b}).\\
            \frac{\|\nabla f(x_k)\|^2}{s_k^2} &= O(s_k^{2a - 2}).
        \end{aligned}
    \end{equation*}
    We choose $C > 0$ such that $C\Lambda_ks_k^{a - b + 1} - |P_k \innp{\zeta_k}{\Gamma_k\hat{r}_k}| > C's_k^{a - 1}$, for some $C' > 0$, which is always possible because $|P_k\innp{\zeta_k}{\Gamma_k\hat{r}_k}| = O(s_{k}^{a - 1})$ and $C\Gamma_ks_k^{a - b + 1} = \Theta(s_{k}^{a - 1})$. Assuming that $s_0$ is small enough so that $\Gamma_k < 1, \forall k \in \NN$, we have:

    \begin{equation*}
        \begin{aligned}
            |\beta_{k+1}| &\leq (1 - \Lambda_k)|\beta_k| + |P_k\innp{\zeta_k}{\Gamma_k\hat{r}_k}| + |(\sin \beta_k - \beta_k)\Lambda_k| + O\left(\frac{\|\nabla f(x_k)\|^2}{s_k^2}\right)\\
            &\leq (1 - \Lambda_k) Cs_k^{a - b + 1} + |P_k\innp{\zeta_k}{\Gamma_k\hat{r}_k}| + O(s_k^{3a - 2b + 1} + s_k^{2a - 2})\\
            &\leq Cs_k^{a - b + 1} - C's_k^{a - 1} + O(s_k^{3a - 2b + 1} + s_k^{2a - 2})\\
            &\leq Cs_{k+1}^{a - b + 1} - C's_k^{a - 1} + O(s_k^{3a - 2b + 1} + s_k^{2a - 2} + s_k^{2a - b}).
        \end{aligned}
    \end{equation*}
    Note that by our choice of $a > b - 1$, we have $a - 1 < \min(3a - 2b + 1, 2a - 2, 2a - b)$. Therefore, with sufficiently small $s_0$, the induction part holds.

    Finally, one has:
    \begin{equation*}
        \lim_{k \to \infty} |\theta_k - \theta_{k + 1}| = \lim_{k \to \infty}\left|\arctan \frac{\rho_k^\theta}{r_k - \rho^r_k}\right| \leq \lim_{k \to \infty} \frac{\|\Gamma_k\nabla f(x_k)\|}{|r_k - \rho^r_k|} = 0.
    \end{equation*}
    That concludes the proof.
\end{proof}

We can conclude the proof of \Cref{theorem:sprial-behavior}, which directly follows from \Cref{lemma:trapping-approaching} and \Cref{lemma:dynamics-angle}.

\end{document}